\documentclass[11pt]{amsart}
\usepackage[utf8]{inputenc}

\RequirePackage[numbers]{natbib}
\usepackage[dvips]{epsfig}
\usepackage{graphicx}
\usepackage{diagbox}
\usepackage{latexsym}
\usepackage{cancel}
\usepackage{ulem}

\usepackage{amsmath,amsthm,amsfonts,amssymb,amscd,amsbsy,dsfont}

\usepackage[colorlinks,linkcolor=red,urlcolor=blue,citecolor=blue]{hyperref}
\usepackage[capitalise]{cleveref}

\usepackage[margin=3cm]{geometry}
\usepackage{aquabox}
\usepackage{stmaryrd}
\usepackage{graphicx}
\usepackage{float}
\usepackage{centernot}

\makeatletter
\def\l@section{\@tocline{1}{12pt plus2pt}{0pt}{}{\bfseries}}

\def\@tocline#1#2#3#4#5#6#7{\relax
    \ifnum #1>-1
  \ifnum #1>\c@tocdepth 
  \else
    \par \addpenalty\@secpenalty%
    \begingroup \hyphenpenalty\@M
    \@ifempty{#4}{%
      \@tempdima\csname r@tocindent\number#1\endcsname\relax
    }{%
      \@tempdima#4\relax
    }%
    \parindent\z@ \leftskip#3\relax \advance\leftskip\@tempdima\relax
    \rightskip\@pnumwidth plus4em \parfillskip-\@pnumwidth
    #5\leavevmode\hskip-\@tempdima #6\nobreak\relax
    \hfil\hbox to\@pnumwidth{\@tocpagenum{#7}}\par
    \nobreak
    \endgroup
  \fi
\fi}
\makeatother
\usepackage{enumitem}
\newcommand{\subscript}[2]{$#1 _ #2$}
\newcommand{\delete}[1]{}

\newcommand{\E}{\espe}
\newcommand{\R}{\RR}

\newcommand{\norm}[1]{\lVert#1\rVert_2}

\newcommand{\rj}{\ln^2(J)}

\newcommand{\mmax}{m_{\max}}
\newcommand{\tdiff}{(t_2 -t_1)}

\newcommand{\Lloc}{L_{\text{loc}}}

\newcommand{\diff}{\mathrm{d}}

  {}
  {\color{red}}%
  {}
  {\color{blue}}%
  {}

\title[Decompounding with unknown noise]{Decompounding with unknown noise through several independent channels}

\author{Guillaume Garnier}
\date{\today}

\begin{document}

\maketitle

\vspace{-0.5cm}
\begin{center}
    Sorbonne Université, Inria, CNRS, Universit{é} de Paris, Laboratoire Jacques-Louis Lions, 4 place Jussieu, 75005
Paris, France
\end{center}
\begin{abstract} In this article, we consider two different statistical models. First,
    we focus on the estimation of the jump intensity of a compound Poisson process in the presence of unknown noise. This problem combines both the deconvolution problem and the decompounding problem. More specifically, we observe several independent compound Poisson processes but we assume that all these observations are noisy due to measurement noise. We construct an Fourier estimator of the jump density and we study its mean integrated squared error. Then, we propose an adaptive method to correctly select the estimator and we illustrate the efficiency of the method with numerical simulations.\\
    Secondly, we introduce in this article the multiplicative decompounding problem. We study this problem with an estimator based on the Mellin transform. We develop an adaptive procedure to select the optimal cutoff parameter. 
\end{abstract}
\vspace{0.5cm}


\noindent \textbf{MSC2020 subject classifications:} 62C12, 62C20, 62G05, 62G07.

\tableofcontents

\section{Introduction}

\subsection{Motivation}
\label{part:introduction}
In the domain of non-parametric statistics, the deconvolution and the decompounding problems are classical and have been widely studied \cite{buchmann2003decompounding, duval2013density, duval2019adaptive, johannes2009deconvolution}.
We give a brief description of the two problems:

\subsubsection*{Deconvolution}
Deconvolution has applications in many fields, such as image processing \cite{kundur1996blind}, microscopy \cite{swedlow2007quantitative}, astronomy \cite{starck2002deconvolution, pantin2017deconvolution},seismology \cite{ulrych1971application, sacchi1997reweighting} and medicine \cite{liu2009deconvolution, ralston2005deconvolution}.

Let $X$ and $\varepsilon$ be two independent random variables with density $f$ and $f_\varepsilon$. The deconvolution problem consists in estimating $f$ from an \iid sample of $Y = X + \varepsilon$, \ie from noisy observations     
    \begin{equation*}
        Y_i = X_i + \varepsilon_i \eqsepv
        i = 1, \ldots, n
        \eqfinv
    \end{equation*}
where $(X_i)_{i = 1}^n$ are \iid random variables with density $f$ and $(\varepsilon_i)_{i = 1}^n$ are, \iid, with density $f_\varepsilon$ and independent of $(X_i)_{i = 1}^n$. \\
This problem have been extensively studied in the literature. The most popular approach estimates $f$ with Fourier estimators \cite{carroll1988optimal, devroye1989consistent,stefanski1990deconvolving, neumann1997effect}. Other approaches have also been developed, for instance by using spline-based methods \cite{averbuch2009spline}, wavelet decomposition \cite{johnstone2004wavelet, starck2003wavelets} and penalization methods \cite{comte2006penalized, comte2007finite}. In most studies, the author assume that the noise density $f_\varepsilon$ is known, however this assumption is not necessary if an additional error sample is available \cite{johannes2009deconvolution, comte2011data}. Many adaptive methods, which only rely on data-driven procedures, have also been developed \cite{griffiths1977adaptive, duval2019adaptive}.
\\

\subsubsection*{Decompounding} In the literature, the term decompounding first appeared in the article of Buchmann and Grübel \cite{buchmann2003decompounding}. 
 Decompounding has many applications in financial mathematics \cite{embrechts2013modelling} and queuing theory \cite{babai2011analysis, gomez2016retrial}. 

Let $(X_k)_{k \in \NN}$ be a family of \iid random variable with density $f$ and $N$ be a homogeneous Poisson process with intensity $\lambda \in (0, \infty)$. Define $Y = (Y_t)_{t \geq 0}$ as the compound Poisson process
    \begin{equation*}
        Y_t
        = \sum^{N_t}_{k = 1}X_k
        \eqsepv
        t \geq 0
        \eqfinp
    \end{equation*}
The decompounding problem consists in estimating $f$ from observations $(Y_{i\Delta}, i=1,\ldots,n)$ of the trajectory of $Y$ over $[0, T]$ at a sampling rate $\Delta > 0$, \ie  at time $(i\Delta,i=1,\ldots,n)$. 
The decompounding problem has been the subject of many articles \cite{bogsted2010decompounding, buchmann2004decompounding, van2004asymptotic}. Most of the time, $f$  is estimated using its characteristic function and Fourier estimators \cite{duval2019adaptive}.  The available data can be observed at high frequencies or at low frequencies \cite{duval2019adaptive, coca2018efficient}. 
The decompounding problem is part of the more general framework of Levy processes that have been widely developed in recent years \cite{reiss2013testing, nickl2012donsker, applebaum2009levy}.
\\

In this article, we study a model that combines both the deconvolution problem and the decompounding problem. In particular, we assume that we observe several compound Poisson processes but that all these observations are noisy due to measurement noise. More precisely, we consider a family of processes
    \begin{equation*}
        Z^j_t = \Bigg( \sum^{N_t^j}_{k = 1}X_k^j \Bigg) + \varepsilon^j_t
        \eqsepv
        t \geq 0
        \eqsepv
        j = 1,\ldots,J
        \eqfinp
    \end{equation*}
We want to construct an estimator of the jump density of the process. Our motivation to study this type of model comes from evolutionary biology.  Many studies aim to estimate the distribution of fitness effects (DFE) of cell. This probability density represents the effect of a new mutation on the fitness of a cell. An accurate determination of the DFE within a population would provide a better understanding of the evolutionary trajectory of the population \cite{eyre2007distribution}.\\

In 2018, Robert et al. \cite{robert2018mutation} developed new experimental methods and a probabilistic model to study the DFE in a population of \textit{Escherichia coli} (\textit{E. coli}). The effect of each mutation on the fitness of a cell is assumed to be drawn according to a random variable $X$ of density $f$ where $f$ denotes the DFE of the cell. 
The authors conclude that the number of mutations in a cell follows a point Poisson process and they estimate the DFE using a method of moment.\\

More specifically, they observe $n = 1476$ cell lines using micro-fluidic methods \cite{wang2010robust}. For each lineage $j = 1, \ldots, n$, it is assume that mutations are deleterious and appear according to a Poisson point process $N^j = (N_t^j)_{t \geq 0}$ with intensity $\lambda > 0$. Each of these mutations modifies the selective value over time $(W_t^j)_{t\geq0}$ of the $j$--lineage. Let us denote $t_i^j$ $(i \in \NN)$ the time of occurrence of the $\{i\} - $th mutation in lineage $j$. The quantity
        \begin{equation*}
             s_i^j = \frac{W_{t_{i-1}^j}^j - W_{t_i^j}^j}{W_{t_{i-1}^j}^j} \;,\; 
            i > 0 \; ,
        \end{equation*}
represents the relative effect of the $\{i\} - $th mutation on the fitness of the individual. \\
Assuming further that these mutations have independent and identically distributed $(s_i)$ effects, then
    \begin{equation*}
    \frac{W_t^j}{W_0^j} = \prod_{i = 1}^{N_t^j}(1 - s_i^j) 
    \eqfinp
    \label{eq:mod_stochastique}
    \end{equation*}

By composing by the logarithm, the evolution of each lineage is controlled by a compound Poisson process, \ie
        \begin{equation*}
            Y_t^j = \ln W_t^j = \sum_{i = 1}^{N_t^j} \ln(1 - s_i^j) 
            \eqfinp
        \end{equation*}

Experimentally, lineages can only be observed at a sampling rate $\Delta > 0$ through discrete observations $(Y_{i\Delta}^j)_{i = 1}^n$. 
Moreover, these observations are noisy due to measurement noise $(\varepsilon^j_{i})_{i = 1}^n$ on the $i$-th observation. In fact, it is only possible to observe the different lineages 
    \begin{equation}\label{eq:main_process}
        Z^j_t = \Bigg( \sum^{N_t^j}_{k = 1}X_k^j \Bigg) + \varepsilon^j_t
        \eqsepv
        t \geq 0
        \eqfinv
    \end{equation}
at sampling times $(i\Delta)_{i = 1,\ldots,n}$, with $X_k^j = \ln(1 - s_k^j)$.  By abuse of notation, we consider that the noise $\varepsilon_t^j$ exists for all $t$ in each lineage $j = 1, \ldots, 1476$. \\

The main goal of this article is to provide a nonparametric estimator that is adapted to estimate the probability density $f$ from the sample of the noisy trajectories and study its asymptotic properties. Such an approach has several advantages. First, it offers a way to study the jump density of a noisy compound Poisson process when we observe several independent processes, which has, to our knowledge, been little studied in the literature. Second this method has the advantage of studying a model very close to the biological situation. \\

However, using a logarithm on the model to transform it into an 'additive' model only gives us information on $\log X$. While it is true that the two models are theoretically equivalent, the additive form forces us to make assumptions about the logarithm of the density we want to reconstruct, what is not easy to justify from modelling perspective \\
To overcome this, we use in \cref{S:multipDecompounding} a method similar to the previous one to study the multiplicative decompounding problem. 

In this setting, we discretely observe one trajectory of a multiplicative compound Poisson process 
    \begin{equation*}
    W_t =\prod_{i = 1}^{N_t}(1 - s_i) 
    \eqfinv
    \end{equation*}
where $(N_t)_{t \geq 0}$ is a Poisson process with intensity $\lambda$ independent of the $\iid$ random variables $(s_i)_i$ with density $f$. \\

The idea of studying 'multiplicative' models was developed in \cite{vardi1989multiplicative} and \cite{vardi1992large} where the authors studied a multiplicative deconvolution problem where the error $U$ is uniformly distributed on $[0,1]$. Such models can be studied using the Mellin transform, which can be seen as a multiplicative version of the Fourier transform. These estimators have been successfully used in  \cite{miguel2021spectral,miguel2023multiplicative} to study the multiplicative deconvolution problem with unknown noise.

\subsection{Main results and organisation of the article}

In \cref{S:decompounding}, we discretely observe $J$ trajectory of the process
 $Z = (Z_t)_{t \geq 0}$ defined in \cref{eq:main_process}, \ie
    \begin{equation*}
        Z_t^j
        = \sum^{N_t^j}_{k = 1}X_k^j + \varepsilon_t^j
        \eqsepv
        t \geq 0
        \eqsepv 
        j = 1,\ldots,J
        \eqfinv
    \end{equation*}
where $(X_k^j)_{k \in \NN}$, $j = 1, \ldots, J$ is a family of \iid random variable with density $f$, $(N^j_t)_{t\geq0}$, $j = 1, \ldots, J$ are homogeneous Poisson processes with intensity $\lambda \in (0, \infty)$ and $(\varepsilon^j_t)_{t \geq 0, }$ are \iid centered random variable such that for every $t,s \geq 0$ and for every $i,j = 1, \ldots, J$ the random variables $\varepsilon^j_t$ and $\varepsilon^i_s$ are independents. Let $\Delta > 0$. We suppose that we observe $$(Z_{i\Delta}^j, i = 1, \ldots n, j = 1, \ldots J).$$ We aim at estimating $f$ from these observations. In \cref{s:setting}, we define an estimator $\hat f_{m,J}$, based on Fourier methods, that rely on two different times $t_1, \, t_2 \in \{i\Delta, i = 1, \ldots, n\}$. We establish in \cref{thm:analysis_error} an upper bound for its $L^2$-risk
\begin{equation*}
    \begin{split}
    \espe \Big(\norm{\hat f_{m, J} - f }^2 \Big)
    \leq  \norm{f_m - f}^2 
		+ \sum_{i = 1}^2 \frac{4 e^{4t_i}}{J (t_2 - t_1)^2} \int_{-m}^{m} \frac{\du}{|\varphi_{\varepsilon}(u)|^2}
		\\     
		+\frac{4 K_{J, t_1, t_2}}{(t_2 - t_1)^2} 
	    \cdot \bigg( \frac{\E[X_i^2]}{J t_i} 
		+ \frac{\E[\noise^2]}{J t_i^2} + 4 \frac{m}{(J t_i)^2} \bigg)
  \eqfinv
    \end{split}
	\end{equation*}
 where $\norm{f} = \Big(\int_{\RR} |f(x)|^2 \diff x \Big)^{1/2}$, $\varphi_\varepsilon$ denotes the characteristic function of $\varepsilon$ and $K_{J, t_1, t_2}$ is a constant that only depends on $J,t_1$ and $t_2$.  We discuss the rates of convergence in \cref{s:speed_cv_one}  depending on the regularity of $f$ and $f_\varepsilon$. In \cref{S:adaptive}, we establish an adaptive procedure to automatically select the value of the threshold $m$ according to the data. \\

In \cref{S:multipDecompounding}, we discretely observe one trajectory of the multiplicative compound process
\begin{equation}
    Y_t = \prod_{i = 1}^{N_t}X_i
    \eqfinv
    \label{multCompound_intro}
\end{equation}
where $(N_t)_{t \geq 0}$ is a Poisson process with constant intensity 
$\lambda > 0$, independent of the $\iid$ random variables $(X_j)_{j \in \NN}$ with common density $f \in L^1([0, \infty)) \cap L^2([0, \infty))$. We suppose that we observe $(Y_{i\Delta}, i = 1, \ldots n)$. We aim at estimating $f$ from these observations. In \cref{sec:mellinDEF}, we recall some properties of the Mellin transform. In \cref{sec:mellin_estimator}, we define the empirical Mellin estimator 
\begin{equation*}
    \hat{\Mcal_c}(t) = \frac{1}{n} \sum_{k = 1}^{N} X_{k}^{c - 1 + it}
    \eqsepv 
    t \in \RR
    \eqfinp
\end{equation*}
and we give some of this properties when it is well defined. In \cref{s:model}, we define an estimator $\hat f_{m, \Delta}$ of $f$, based on the Mellin transform. In \cref{sec:risk_mellin}, we establish in \cref{T:risk} an upper bound for its $L^2$-risk
\begin{equation*}
\resizebox{.95\hsize}{!}{$
    \EE \big[ \|\hat{f_{m, \Delta}}  - f \|_{\omega_1}^2\big] \leq \|f_m  - f \|_{\omega_1}^2 
    +
    \frac{1}{2\pi n \Delta^2} \int_{-m}^{m}\frac{1}{|\Mcal_1[\Delta](s)|^2}\diff s 
    + \frac{25}{\pi} \Big(\frac{\EE[\ln(X_1)^2]}{\Delta n} + 4 \frac{m}{(n \Delta)^{2}} \Big)
    \eqfinp
$}
\end{equation*}
where $
    \| f \|_{\omega_1} = \Big( \int_{0}^\infty x |f(x)|^2 \diff x \Big)^{1/2}
    \eqfinp$
In \cref{sec:mellin_adaptive}, we introduce an adaptive procedure $m$ to select the optimal cutoff parameter for the Mellin estimator. \\

We numerically illustrate our two methods on several examples in \cref{S:simu}.
The proofs are postponed to the \cref{sec:proof}. \\

We rigorously define the distinguished logarithm which is a key element of our estimators in \cref{S:DistingLog} and we describe some of its classical properties. We present some useful lemmas in \cref{s:alem}.\\

All of our computer code is available and documented on GitHub
 \begin{center}
\url{https://github.com/guimgarnier/decompounding-with-noise}
\end{center}
with several examples, making it easy to use on biological experimental data.\\

\section{Decompounding with unknown noise}
\label{S:decompounding}

\subsection{Notation}
\label{part:notation}

\subsubsection*{The empirical characteristic function}

The characteristic function of a real random variable $X$ is
	\begin{equation*}
	\varphi_X(u) = \int_\RR e^{i u x}\PP_X(dx) =\EE[e^{iuX}]
    \eqsepv u \in \RR
	\eqfinp
	\end{equation*}	
If $X$ has a density $f \in L^1(\RR) \cap L^2(\RR)$, then we recall the inversion formula
\begin{equation*}
    f(x) = \frac{1}{2\pi}\int_\RR \varphi_X(u) e^{-iux} \diff u
    \eqsepv
    x \in \RR
    \eqfinv
\end{equation*}
and Parseval's identity 
\begin{equation*}
    \int_\RR |f(x)|^2 \diff x = \frac{1}{2\pi}\int_\RR |\varphi_X(u)|^2 \diff u
    \eqfinp
\end{equation*}

Let $(X_1, \ldots, X_N)$ be \iid real-valued random variables with common characteristic function $\varphi_X(t)$. The function
    \begin{equation*}
        \hat \varphi_X^N(u) = \frac{1}{N} \sum_{k=1}^N e^{iuX_k}
        \eqsepv
        u \in \RR 
        \eqfinv
    \end{equation*}
is called the empirical characteristic function associated with the sample $(X_1, \ldots, X_N)$. \\

The empirical characteristic function is an unbiased estimator of $\varphi_X$ \ie $$\EE [\hat \varphi_X^N(u)] = \varphi_X(u)$$ and
\begin{equation}\label{varECF}
    \EE [|\hat \varphi_X^N(u) - \varphi_X(u)|^2] = \frac{1}{N}(1 - |\varphi_X(u)|^2)
    \eqfinp
\end{equation}
For more detailed information, we refer to Ushakov's book \cite{ushakov2011selected}.

\subsubsection*{The distinguished logarithm} In the following, the notation $\log$ denotes the distinguished logarithm that we rigorously define in Appendix \ref{S:DistingLog}.

\subsection{Statistical setting}
\label{s:setting}
Let $f \in L^1(\RR) \cap L^2(\RR)$. Let $(X_i^j)_{i, j \geq 0}$ be \iid real-valued random variables with density $f$ and for some $J \in \NN$, let $(N_t^j)_{t\geq0}, j = 1,\ldots,J $ be a family of \iid Poisson point processes with intensity $\lambda \in (0, \infty)$, independent of $(X_i^j)_{i, j \geq 0}$. \\

We consider a family $(Y_t^j)_{t\geq0}$, $j = 1\ldots,J$  of compound Poisson processes with intensity $\lambda$ and jump size density $f$,
	\begin{equation*}
		Y_t^j = \sum_{k = 1}^{N_t^j} X_k^j
		\eqsepv
		t \geq 0
		\eqsepv
		j = 1,\cdots,J
		\eqfinp
	\end{equation*}

It is assumed that the observations are disturbed by a random noise independent of the observed process, \ie we observe a process of the form
    \begin{equation*}
        Z^j_t = Y^j_t + \varepsilon^j_t
        = \BBp{\sum^{N_t^j}_{k = 1}X_k^j} + \varepsilon^j_t
        \eqsepv
        t \geq 0
        \eqfinv
    \end{equation*}
for every $i,j = 1, \ldots, J $. For every reals $t,s \geq 0$, the random variables $\varepsilon^j_t$ and $\varepsilon^i_s$ are independents. \\

Also, we set the following assumptions:

\begin{hypothesis*}\text{}
    \begin{enumerate}[label=(\subscript{H}{{\arabic*}})]
        \item $\fctcar{X_1} \in L^1(\RR)$.
        \label{hyp:target_fctcar_is_invertible}
        \item $\espe(X_1^2) < \infty$.  \label{hyp:target_is_square_integrable}
        \item $\forall t \in [0, \infty) \eqsepv \espe[\noise_t] = 0$ 
        \label{hyp:noise_is_centered_and_square_integrable}
        and $\espe[\noise_t^2] <\infty$.
        \item  $\forall m \in [0, \infty) \eqsepv \exists c_m \in (0, \infty) : \forall u \in [0, m]
        \eqsepv |\fctcar{\varepsilon}(u)| \geq c_m$. \label{hyp_noise_is_bounded}
        \item  $\forall u \in \RR
        \eqsepv |\fctcar{\varepsilon}(u)| > 0$. \label{hyp_noise_no_vanish}
    \end{enumerate}
\end{hypothesis*}

For the sake of simplicity, we write $\varepsilon$ instead of $\varepsilon_t^j$ in the rest of this article.

\subsubsection*{Construction of the estimator.}
We estimate the characteristic function $\varphi_X$ of $X$. The main idea is to claim that if we have a good reconstruction of $\varphi_X$, by applying the inverse Fourier transform, we should have a good reconstruction of $f$. 

For all $t \in [0, \infty)$,  the characteristic function of the process on a single channel $Z_t^j$ is given by
    \begin{equation}
         \varphi_{Z_t}(u) = e^{- \lambda t + \lambda t \varphi_X(u)} 
        \cdot \varphi_\noise(u)
        \eqsepv
        \forall u \in \RR
        \eqfinp
        \label{eq:fctcar_un_canal_bruit}
    \end{equation}

    To simplify notation, we consider in the rest of the article the particular case where $\lambda = 1$. \\

Consider two different times $0 < t_1 < t_2$, then
    \begin{equation*}
        \frac{\varphi_{Z_{t_2}}}{\varphi_{Z_{t_1}}}
        = 
        e^{
            - \np{t_2 - t_1} 
            + \np{t_2 - t_1} \varphi_X(u)}
    \eqsepv
    \forall u \in \RR
    \eqfinp
    \end{equation*}

Applying the distinguished logarithm, we obtain the explicit formula
    \begin{equation*}
        \varphi_X (u)
        = 1 + \frac{1}{t_2 - t_1} \Bc{ \log \varphi_{Z_{t_2}} (u) - \log \varphi_{Z_{t_1}} 		        (u)}
    \eqsepv
    \forall u \in \RR
    \eqfinp
    \label{reconstruction_exacte_fctcar}
    \end{equation*}

This leads us to consider the estimator
    \begin{equation*}
        \forall u \in \RR
        \eqsepv
        \hat \varphi_X^{\, J} (u)
        = 1 + \frac{1}{t_2 - t_1} \Bc{ \log \hat \varphi_{Z_{t_2}}^{\, J} (u) - \log 		   
        \hat \varphi_{Z_{t_1}}^{\, J} (u)}
        \eqfinv
    \end{equation*}
    with for all $u \in \RR$, $\tau \in \{t_1, t_2\}$
    \begin{equation*}
    	\hat \varphi_{Z_{\tau}}^{\; ' J} (u) = \frac{1}{J} \sum_{j = 1}^J i Z_{\tau}^j e^{i u Z_{\tau}^j}
    \eqsepv
    \hat \varphi_{Z_{\tau}}^{\, J} (u) = \frac{1}{J} \sum_{j = 1}^J e^{i u Z_{\tau}^j}
    \eqsepv
    \log \hat \varphi_{Z_{\tau}}^{\, J} (u) = \int_0^u \frac{\hat \varphi_{Z_{\tau}}^{\; ' J} \np{z}}{\hat \varphi_{Z_{\tau}}^{\, J} \np{z}} \dz
    \eqfinp
    \label{def:estimateur_fctcar}
    \end{equation*}

Since $\varphi_X$ is a characteristic function, its modulus is bounded by 1. Nevertheless, this is not necessarily the case for $\hat \varphi_X^{\, J}$. We avoid this explosion problem taking
    \begin{equation*}
        \tilde \varphi_X^{\, J} (u) 
        = 1 + \frac{1}{t_2 - t_1} \Bc{ \log \hat \varphi_{Z_{t_2}}^{\, J} (u) \cdot \1{|\log 		   
        \hat \varphi_{Z_{t_2}}^{\, J} (u)| \leq \ln(J)} - \log 		   
        \hat \varphi_{Z_{t_1}}^{\, J} (u) \cdot \1{|\log 		   
        \hat \varphi_{Z_{t_1}}^{\, J} (u)| \leq \ln(J)}}
        \eqsepv
        u \in \RR
        \eqfinp
    \end{equation*}
To simplify notation, for any positive real $\tau > 0$ and for all $u \in \RR$ we set
    \begin{equation*}
    \log \tilde \varphi_{Z_{\tau}}^{\, J}  (u) 
    =
    \log \hat \varphi_{Z_{\tau}}^{\, J} (u) \cdot \1{|\log 		   
        \hat \varphi_{Z_{\tau}}^{\, J} (u)| \leq \ln(J)}
    \eqfinp
    \end{equation*}

In particular, with this notation
    \begin{equation}
    \label{def:estimateur_fctcar_tilde}
        \tilde \varphi_X^{\, J} (u) 
        = 1 + \frac{1}{t_2 - t_1} \Big[ \log \tilde \varphi_{Z_{t_2}}^{\, J} (u) - \log 		   
        \tilde \varphi_{Z_{t_1}}^{\, J} (u) \Big]
        \eqsepv
        u \in \RR
        \eqfinp
    \end{equation}

Now, we find an estimator of $f$ by performing an inverse Fourier transformation. However, $\tilde \varphi_X^{\, J}$ is not necessarly integrable. Hence, we eliminate frequencies above a threshold $m$ before applying an inverse Fourier transform, we obtain
	\begin{equation}
        \hat f_{m, J}(x) = \frac{1}{2\pi} \int_{-m}^m e^{-iux} \tilde \varphi_X^{\, J} (u) \du
        \eqsepv
        x \in \RR
        \eqsepv
        m \in (0, \infty)
        \eqfinp
    \label{def:estimateur}
    \end{equation}

In the same way, we perform an truncated inverse Fourier transformation from the true characteristic function of $X$ and we define
	\begin{equation*}
        f_{m}(x) = \frac{1}{2\pi} \int_{-m}^m e^{-iux} \varphi_X (u) \du
        \eqsepv
        x \in \RR
        \eqsepv m \in (0, \infty)
        \eqfinp
    \label{def:true_estimateur_truncated}
    \end{equation*}

\subsection{Risk bounds}
\label{sec:risk_decompound}
In this section, we study the mean integrated squared error (MISE) of the estimator $\hat f_{m, J}$. \\

For any $0 < t_1 < t_2$, we define
	\begin{equation}\label{def_Cj}
	C^J_{t_1, t_2} = \min \Big\{ m \geq 0 \, \Big| \, 3t_2 - t_1 + \sup_{[-m, m]}| \log \varphi_{\varepsilon}(\cdot)| > \ln(J) \Big\} 
	\eqfinp
	\end{equation}	

\begin{theorem}\label{thm:analysis_error}
Suppose that the assumptions \ref{hyp:target_fctcar_is_invertible} -- \ref{hyp_noise_no_vanish} holds.
Let $0 < t_1 < t_2 $. We suppose that $J$ is enough large such that for all $i = 1,2$:
\begin{enumerate}
    \item $t_i < \frac{1}{4} \log(J t_i)$,
    \item $\sqrt{\log(J t_i)} (J t_i)^{2\delta_i-1/2} < 1$ with $\delta_i = t_i / \log(J t_i)$.
\end{enumerate}
Then for any $m < C^J_{t_1, t_2}$ with $ C^J_{t_1, t_2}$ defined by \eqref{def_Cj},we have
	\begin{equation*}
    \begin{split}
    \espe \Big(\norm{\hat f_{m, J} - f }^2 \Big)
    \leq  \norm{f_m - f}^2 
		+ \sum_{i = 1}^2 \frac{4 e^{4t_i}}{J (t_2 - t_1)^2} \int_{-m}^{m} \frac{\du}{|\varphi_{\varepsilon}(u)|^2}
		\\     
		+\frac{4 K_{J, t_1, t_2}}{(t_2 - t_1)^2} 
	    \cdot \bigg( \frac{\E[X_i^2]}{J t_i} 
		+ \frac{\E[\noise^2]}{J t_i^2} + 4 \frac{m}{(J t_i)^2} \bigg)
	\eqfinp	
    \end{split}
	\end{equation*}
	where $K_{J, t_1, t_2} = m \rj
	 + 
	 8 m t_2^2
	 +
	 \int\limits_{-m}^m | \log\varphi_\varepsilon (u)|^2 \du $.
\end{theorem}

\begin{remark} Theorem  \ref{thm:analysis_error} is asymptotic and ensures that the variance term vanishes when $J \to \infty$. The upper bound obtained is the sum of a bias term and of a variance term. \\
In the variance, there is a term $V \sim \frac{4 e^{4t_i}}{J (t_2 - t_1)^2} $. On one side, the presence of $e^{4t_i}$ means that the estimator is more and more imprecise as we look the sample at a very large time $t_2$. On the other side, the presence of $(t_2 - t_1)$ at the denominator means that we cannot take $t_1$ and $t_2$ too close to each other. We illustrate this trade-off in Section \ref{S:simu} with numerical simulations.
\end{remark}

\idea{\begin{remark} For any positive real $\tau > 0$, we have $\varphi_{Z_{\tau}} = \varphi_{Y_{\tau}} \cdot \varphi_\varepsilon$ from Proposition \ref{prop:fctcar_compound_poisson_without_noise} and Equation \eqref{eq:fctcar_un_canal_bruit}. By using the classical inequality $2ab \leq {a^2 + b^2}$ for any reals $(a,b) \in \RR^2$, we have that
	\begin{equation*}
	\forall \tau \in (0, \infty)
	\eqsepv
	\frac{4}{J \tau}  \int_{-m}^{m} \frac{\du}{|\varphi_{Z_{\tau}}(u)|^2}
	\leq 
	\frac{2}{J \tau}  \int_{-m}^{m} \frac{\du}{|\varphi_{Y_{\tau}}(u)|^4} 
	+ \frac{2}{J \tau}  \int_{-m}^{m} \frac{\du}{|\varphi_\varepsilon(u)|^4}  
	\eqfinp
	\end{equation*}
\end{remark}}

\subsection{Speed of convergence}
\label{s:speed_cv_one}

In this section, we study the optimal choice of $m$ based on the regularity of the jump density and regularity of the noise density, \ie that when $J \to \infty$, we look at the asymptotic behaviour of $m$ when it minimizes the upper term in Theorem in \ref{thm:analysis_error}. We organize the discussion according to different regularities of $f$ and $f_\varepsilon$. More specifically, we consider:
	\begin{itemize}
		\item ordinary smooth densities: The characteristic function decays as $|u|^{-2a}$ (example: Gamma distribution)
		\item super smooth densities: The characteristic function decays as $e^{-|u|^s}$ (example: Cauchy distribution)
	\end{itemize}

\subsubsection*{When $f$ is ordinary smooth and $f_\varepsilon$ is ordinary smooth}\label{sec:smoothsmooth}

In this case, we assume that $f \in \Scal(\beta, L)$ for some $\beta > 0$, $L > 0$, with
	\begin{equation*}
	\Scal(\beta, L) 
	= 
	\Big\{ f \in L^2(\RR) \, \Big| \, \int_{\RR} (1 + |u|^2)^{\beta} |\varphi_X(u)|^2 du \leq L \Big\}
	\eqfinp
	\end{equation*}
We also assume that $f_\varepsilon$ is \textit{ordinary smooth}, \ie there exists two reals $a > \frac{1}{2}$ and $d > 0$, such that 
	\begin{equation*}
	\forall u \in \R 
	\eqsepv
	d \leq (1+u^2)^a |\varphi_\varepsilon(u)|^2 \leq \frac{1}{d}
	\eqfinp
	\end{equation*}

It follows that $\norm{f - f_m}^2$ is of order $m^{-2\beta}$ and that  the variance term $\frac{1}{J} \int_{-m}^{m} \frac{\du}{|\varphi_{\varepsilon}(u)|^2}$ is of order $\frac{1}{J}  m^{2a + 1}$.
	
Moreover, the variance term $\frac{4 K_{J, t_1, t_2}}{(t_2 - t_1)^2} 
	 \cdot \bigg( \frac{\E[X_i^2]}{J t_i} 
			+ \frac{\E[\noise^2]}{J t_i^2} + 4 \frac{m}{(J t_i)^2} \bigg)$ is of order $
        \frac{m \rj}{J}$.		
Therefore, if the Assumptions \ref{hyp:target_fctcar_is_invertible} -- \ref{hyp_noise_no_vanish} hold, then
	\begin{equation}
	\EE \Big[ \norm{\hat f_{m, J} - f}^2 \Big]
	\lesssim
	m^{-2\beta} + \frac{1}{J} m^{2a + 1} +  m \cdot \frac{\ln^2(J)}{J}
	\eqfinv
	\label{E:speed_ordord}
	\end{equation}
up to a multiplicative constant that depends on $t_1$ and $t_2$. \\

To obtain a value $m^\star$ that reaches the bias-variance compromise, we minimize the upper bound of Equation~\eqref{E:speed_ordord} by differentiating with respect to $m$. It follows that $m^\star \sim J^{\frac{1}{2a + 2\beta + 1}}$ when $J \to \infty$ and that 
	\begin{equation*}
	\EE \Big[ \norm{\hat f_{m, J} - f}^2 \Big]  = O\left( J^{\frac{-2\beta}{2\beta + 2a + 1}} \right)
	\eqsepv
	J \to \infty
	\eqfinp
	\end{equation*}
\subsubsection*{When $f$ is ordinary smooth and $f_\varepsilon$ is super smooth smooth}\label{sec:smoothsuper}
In this case, we assume that $f \in \Scal(\beta, L)$ for some reals $\beta > 0$, $L > 0$, \ie
	\begin{equation*}
	f \in \Scal(\beta, L) 
	= 
	\Big\{ f \in L^2(\RR) \, \Big| \, \int_{\RR} (1 + |u|^2)^{\beta} |\varphi_X(u)|^2 du \leq L \Big\}
	\eqfinp
	\end{equation*}
We also assume that $f_\varepsilon$ is \textit{super smooth}, \ie there exists $a > 0, s > 0$  and $d_1, d_2 > 0$ such that 
	\begin{equation*}
	\forall u \in \R 
	\eqsepv
	d_1 \leq \exp \big( b \cdot |u|^s \big) \cdot |\varphi_\varepsilon(u)|^2 \leq d_2
	\eqfinp
	\end{equation*}
	
It follows that $\norm{f - f_m}^2$ is of order $ m^{-2\beta}$ and that  the variance term $\frac{1}{J} \int_{-m}^{m} \frac{\du}{|\varphi_{\varepsilon}(u)|^2}$ is of order $\frac{1}{J} m \cdot e^{b\cdot m^s}$.
	
Moreover, the variance term $\frac{4 K_{J, t_1, t_2}}{(t_2 - t_1)^2} 
	 \cdot \bigg( \frac{\E[X_i^2]}{J t_i} 
			+ \frac{\E[\noise^2]}{J t_i^2} + 4 \frac{m}{(J t_i)^2} \bigg)$ is of order $m \cdot \frac{\ln^2(J)}{J}$.\\

Therefore, if the Assumptions of Theorem \ref{thm:analysis_error} holds, then 
	\begin{equation}
	\EE \norm{\hat f_{m, J} - f}^2  
	\lesssim 
	m^{-2\beta} + \frac{1}{J} m \cdot e^{b\cdot m^s} + m \cdot \frac{\ln^2(J)}{J}
	\eqfinv
	\label{E:speed_ordsmooth}
	\end{equation}
 up to a multiplicative constant that depends on $t_1$ and $t_2$.

To obtain a value $m^\star$ that reaches the bias-variance compromise, we minimize the upper bound of Equation~\eqref{E:speed_ordsmooth} by differentiating with respect to $m$. It follows that $m^\star \sim \big( \frac{\ln(J)}{b} \big)^{1/s}$ when $J \to \infty$ and that
	\begin{equation*}
	\EE \norm{\hat f_{m, J} - f}^2  
	= O\left( \Big(\frac{\ln(J)}{b}\Big)^{\frac{-2\beta}{s}} \right)
	\eqfinp
	\end{equation*}
	
\subsubsection*{When $f$ is super smooth} 
In this section, we assume that $f$ belongs to the class of super smooth densities, 
	\begin{equation*}
	\Acal_{c, s}(L) 
	= 
	\Big\{ f \in L^2(\RR) \, \Big| \, \int_{\RR} \exp(c \cdot|u|^s) |\varphi_X(u)|^2 du \leq L \Big\}
	\eqfinp
	\end{equation*}
It follows that $\norm{f - f_m}^2 $ is of order $ \exp(-c\cdot|m|^s)$. \\

We apply the same strategy that in \cref{sec:smoothsmooth} and \cref{sec:smoothsuper} to compute $m^\star$ when $f_\varepsilon$ is ordinary smooth or super smooth. \\

In \cref{table:speed}, we resumes the different results that we obtained with respect to the regularity of $f$ and $f_\varepsilon$. For each case, we give the optimal cut-off $m^*$ and the speed of convergence of the estimator. 
 
\begin{table}
\centering
\resizebox{\columnwidth}{!}{%
\begin{tabular}{c||c|c|}
 \diagbox{$f$}{$f_\varepsilon$} & ordinary smooth: $f_\varepsilon \in \Scal(a, L)$ & super smooth: $f_\varepsilon \in \Acal_{b, s}(L)$ \\
 \hline
 \hline
 ordinary smooth: $f \in \Scal(\beta, L)$
 	&  \begin{tabular}{c} $m^\star \sim J^{\frac{1}{2a + 2\beta + 1}}$ \\
	$
	\EE \norm{\hat f_{m, J} - f}^2  = O\left( J^{\frac{-2\beta}{2\beta + 2a + 1}} \right)$ \end{tabular} & 
	 \begin{tabular}{c} $m^\star \sim \frac{\ln(J)}{b}$ 
 	\\ $\EE \norm{\hat f_{m, J} - f}^2  
	= O\left( \Big(\frac{\ln(J)}{b}\Big)^{\frac{-2\beta}{s}} \right)$ \end{tabular}\\
 \hline
 super smooth 
 $f \in \Acal_{c, s}(L)$
 	&  
 	\begin{tabular}{c} $m^\star \sim \Big(\frac{\ln(J)}{c}\Big)^{\frac{1}{s}}$ 
 	\\ $\EE \norm{\hat f_{m, J} - f}^2  
	=O ( J^{-1} )$ \end{tabular}
	& 
	 \begin{tabular}{c} $m^\star \sim \Big(\frac{\ln(J)}{b+c}\Big)^{\frac{1}{s}}$ 
 	\\ $\EE \norm{\hat f_{m, J} - f}^2  
	= O \Big( J^{\frac{-c}{b+c}} \Big)$ \end{tabular}\\
 \hline
\end{tabular}
}
\caption{Speed of the optimal upper bound with respect to the regularity of $f$ and $f_\varepsilon$.}
\label{table:speed}

\end{table}

\subsection{The adaptive procedure}
\label{S:adaptive}

We have constructed in \cref{sec:risk_decompound} a statistical estimator $\hat f_{m, J}$ to estimate the jump size density $f$.
However, as this estimator strongly depends on the choice of the parameter $m$, we would like to be able to select a value of $m$ that depends only on the available data, without a priory knowledge on the regularity of the density $f$. \\

To do this, we aim to select the parameter $m$ that minimizes the bound obtained in the Theorem \ref{thm:analysis_error}.\\

For this, we need to make some further regularity assumptions on the measurement noise, which is to say that there is a function $g \in L^1(\RR)$ and a real $d > 0$ such that 
    \begin{equation*}
        \forall u \in \R 
	    \eqsepv
        d \cdot g(u) < |\varphi_\varepsilon(u)|^2 < \frac{g(u)}{d}
        \eqfinp
    \end{equation*}

In the rest of this section, we assume that the noise is ordinary smooth, \ie there exists two reals $a > \frac{1}{2}$ and $d \in (0, 1)$, such that 
	\begin{equation*}
	\forall u \in \R 
	\eqsepv
	d \leq (1+u^2)^a \cdot |\varphi_\varepsilon(u)|^2 \leq \frac{1}{d}
	\eqfinp
	\end{equation*}
Nevertheless, the methods used can be adapted for any function $g$. \\

In the following, we simplify the computations by considering the particular case $a = 1$, \ie we assume that 
 	\begin{equation*}
	\forall u \in \R 
	\eqsepv
	d \leq (1+u^2) \cdot |\varphi_\varepsilon(u)|^2 \leq \frac{1}{d}
	\eqfinp
	\end{equation*}

\subsubsection*{How to select $\hat m_{J,t_2,t_1}$}

As we said above, we aim to select $m$ so as to minimize the bound obtained in the theorem \ref{thm:analysis_error}. \\

The dominant terms in this bound are the bias term $\int_{u \in [-m, m]} |\varphi_X(u)|^2 \du$ and the variance term $ \frac{4 e^{4t_2}}{J (t_2 - t_1)^2} \int_{-m}^{m} \frac{\du}{|\varphi_{\varepsilon}(u)|^2}$. Through differentation, the optimal $\bar{m_J}$ satisfies
    \begin{equation*}
        |\varphi_X(\bar{m_J})|^2 = \frac{4a  e^{4t_2}}{J (t_2 - t_1)^2} (1 + {\bar{m_J}}^{2})
        \eqfinv
    \end{equation*}

then
    \begin{equation*}
       \bigg|\frac{\varphi_X(\bar{m_J})}{\sqrt{(1 + \bar{m_J}^{2})}}  \bigg|^2 = \frac{4a e^{4t_2}}{J (t_2 - t_1)^2}
        \eqfinp
    \end{equation*}

However we do not know $\varphi_X$, so it is impossible to calculate directly $\bar{m_J}$.
Following the strategy developed by Duval and Kappus \cite{duval2019adaptive}, we consider
    \begin{equation*}
    \bar \varphi_X^J(u) = \tilde \varphi_X^J(u) \cdot \1{ \big|\frac{\tilde \varphi_X^J(u)}{\sqrt{1 + u^2}}  \big| \geq \frac{\kappa_{J, t_1, t_2}}{\sqrt{J}(t_2 - t_1)}}
    \eqfinv
    \end{equation*}
where $\kappa_{J, t_1, t_2} = 2 e^{2 t_2} + \kappa \sqrt{\ln(J ( t_2 - t_1)^2)}$, $\kappa > 0$. \\

It leads us to define the empirical cutoff parameter
    \begin{equation*}
        \hat m_J = \max \bigg\{ u \geq 0: \Big|\frac{\bar \varphi_X(u)}{\sqrt{1 + u^2}}  \Big| \geq \frac{\kappa_{J, t_1, t_2}}{\sqrt{J}(t_2 - t_1)}\bigg\} \wedge \big(J(t_2 - t_1)^2 \big)^\alpha
        \eqsepv
        \alpha \in (0, 1)
        \eqfinp
    \end{equation*}

For simplicity, we note $\mmax = \big(J(t_2 - t_1)^2 \big)$, keeping in mind that it depends on $J, t_1$ and $t_2$. \\

We define a new estimator 
	\begin{equation*}
        \bar f_{m, J}(x) = \frac{1}{2\pi} \int_{-m}^m e^{-iux} \bar \varphi_X^{\, J} (u) \du
        \eqsepv
        x \in \RR
        \eqfinp
    \end{equation*}
    
\begin{theorem} 
Assume \ref{hyp:target_fctcar_is_invertible}--\ref{hyp_noise_is_bounded}. For all reals $0 < t_1 < t_2 $ such that $t_2 \leq \frac{1}{4} \log(J t_2)$ and  $(\mmax)^\alpha < C^J_{t_1, t_2}$, 
$Jt_1 \to \infty \eqsepv Jt_2 \to \infty$ as $J \to \infty$. 
Then,
    \begin{align*}
        \EE \big[ \norm{\bar f_{\hat m_J} - f }^2 \big] 
        & \leq 
        \inf_{m \in [0, (\mmax)^\alpha]}
        \Big\{
        \norm{f_m - f}^2  + C \frac{\ln(J ( t_2 - t_1)^2) \cdot m \cdot(1 + m^2) }{J ( t_2 - t_1)^2} + \tilde C A 
        \Big\} \\
        & \qquad \qquad \qquad \qquad
        + \Big(2 + \frac{2 \log(J)}{(t_2 - t_1)}\Big)^2 \cdot T_J
    \end{align*}
where 
    \begin{equation*}
        A = \sum_{i = 1}^2 \frac{4 e^{4t_i}}{J (t_2 - t_1)^2} \int_{-m}^{m} \frac{\du}{|\varphi_{\varepsilon}(u)|^2}
		+\frac{4 K_{J, t_1, t_2}}{(t_2 - t_1)^2} 
	 \cdot \bigg( \frac{\E[X_i^2]}{J t_i} 
			+ \frac{\E[\noise^2]}{rJ t_i^2} + 4 \frac{m}{(J t_i)^2} \bigg)
    \end{equation*}
and 
    \begin{equation}
        T_J \leq C_0 (J \tdiff^2)^{\alpha - c(\theta)^2} + \frac{C_1}{J \tdiff^2} + \frac{C_2}{J \tdiff^4}
    \end{equation}
and 
 $c(\theta) = \kappa (t_2 - t_1)  e^{2 t_2} \cdot \frac{d}{\sqrt{1 + (\mmax)^2}}$ and where $C_0, C_1$ and $C_2$ depends on $\EE[X_1^2], \EE[\varepsilon^2]$ and where $C$ and $\Tilde{C}$ are two constants.
\label{thm:oracle}
\end{theorem}

\section{Multiplicative decompounding}\label{S:multipDecompounding}

\subsection{Preliminaries on Mellin transform} 
\label{sec:mellinDEF}
 In this section, we first recall some classical results on Mellin transform. We define the multiplicative decompounding problem and we define a good non-parametric estimator.  

Let $\mu$ be a probability measure on $(0, \infty)$. Its Mellin transform is
    \begin{equation}
    \Mcal[\mu](s) = \int_0^\infty x^{s-1} \mu(\diff x)
    \eqsepv 
    s \in \CC
    \eqfinv
    \label{def:mellin}
    \end{equation}
for those values of $s$ for which this integral is well defined.\\
\begin{remark}
    If $\Mcal[\mu](u)$ is well defined for some $u \in \RR$, then $\Mcal[\mu]$ converges on the vertical line $u + i\RR$. Furthermore, if the integral exists for $u$ and $v$ in $\RR \; ( u < v)$,then $\Mcal[\mu](w)$ exists for $w \in (u,v)$. It follows that the Mellin transform of a probability measure is well defined on a vertical band $\Xi_\mu$ of the complex plane. The region $\Xi_\mu$ is called the fundamental strip of $\Mcal[\mu]$. 
\end{remark}
Assume that $\mu$ has a density $f \in \Lloc^1([0, \infty))$, then we denote
    \begin{equation*}
    \Mcal[f](s):= \Mcal[f\diff x](s) = \int_0^\infty x^{s-1} f(x) \diff x
    \eqsepv 
    s \in \CC
    \eqfinv
    \end{equation*}
for those values of $s$ for which this integral is well defined. We denote $\Xi_f$ the fundamental strip of $\Mcal[f\diff x]$\\

Let $f \in \Lloc^1([0, \infty))$ and $c \in \RR$ that belongs to $ \Xi_f$. Then we define
    \begin{equation*}
        \Mcal_c[f](t):= \Mcal[f](c + it) = \int_0^\infty x^{c-1+it}f(x) \diff x
        \eqsepv
        t \in \RR
        \eqfinp
    \end{equation*}
We observe that the function $x \mapsto x^{c-1}f(x)$ belongs to $L^1([0, \infty))$. \\

We define the norm 
\begin{equation*}
    \| f \|_{\omega_c} = \Big( \int_{0}^\infty |f(x)|^2 x^{2c-1} \diff x \Big)^{1/2}
    \eqfinp
\end{equation*}
and $ L^2_{[0, \infty)}(\omega_c)$ denotes the space of function $f: [0, \infty) \to \RR$ such that $\| f \|_{\omega_c} < \infty$\\

Let $f$, $g \in \Lloc^1([0, \infty))$. The multiplicative convolution of $f$ and $g$ is the function defined by 
\begin{equation*}
    f \star g (x) = \int_0^\infty f(y) g\Big(\frac{x}{y}\Big) \diff y
    \eqsepv
    x \in [0, \infty)
    \eqfinp
\end{equation*}

\begin{proposition}[Multiplicative convolution and Mellin transform] 
\label{P:mellinConvol}
Let $\mu$ and $\nu$ be two probability measures on $[0, \infty)$. We have
    \begin{equation*}
        \Mcal[\mu \star \nu](s) = \Mcal[\mu](s) \Mcal[\nu](s)
    \eqsepv 
    s \in \CC
    \eqfinv
    \end{equation*}
whenever this integral is well defined.\\
\end{proposition}

\begin{proposition}[Parseval's theorem for Mellin transform]
\label{P:melParseval}
    \begin{equation}
        \int_0^\infty \frac{|f(x)|^2}{x} \diff x 
        = 
        \frac{1}{2\pi} \int_0^\infty \Mcal[f](it) \diff t
        \eqfinp
    \end{equation}
\end{proposition}

\begin{proposition}(Inversion formula)
\label{P:inv_mellin} Let $f \in L^1([0, \infty))$ and $c \in \Xi_f$. Then 
    \begin{equation*}
    f(x) = \Mcal^{-1} \Mcal[f](x) =  \frac{1}{2\pi i }\int_{\nu -i \infty}^{\nu + i\infty} x^{-s}  \Mcal[f](s) \diff s
    \eqsepv
    x \in [0, \infty)
    \eqfinp
    \end{equation*}
\end{proposition}

\subsection{The Mellin estimator}
\label{sec:mellin_estimator}
Let $f \in L^2_{[0, \infty)}(\omega_c)$ be a probability distribution. Let us assume that we have a sample $X_1, \ldots,X_N$ independent and identically distributed according to $f$. \\

We define the Mellin estimator of $f$
\begin{equation*}
    \hat{\Mcal_c}(t) = \frac{1}{n} \sum_{k = 1}^{N} X_{k}^{c - 1 + it}
    \eqsepv 
    t \in \RR
    \eqfinp
\end{equation*}

We use Proposition \ref{P:inv_mellin} to define the estimator. 
\begin{equation*}
    \hat{f_m}(x) = \frac{1}{2\pi} \int_{-m}^m x^{-c - it} \hat{\Mcal_c}(t) \diff t
    \eqfinv
\end{equation*}
which is an estimator of 
\begin{equation*}
    f_m(x) = \frac{1}{2\pi} \int_{-m}^m x^{-c - it} \Mcal_1(t) \diff t
    \eqfinp
\end{equation*}

The next proposition give a bound of the $L^2_{[0, \infty)}(\omega_c)$-risk of $\hat{f_k}$. 
\begin{proposition}[Miguel et al.  \cite{miguel2021spectral}
Proposition 2.1]
    If $f \in L^2_{[0, \infty)}(\omega_c)$ and $\sigma_c^2:= \EE_f(X^{2(c-1)}) < \infty$, then for all $k \in [0, \infty)$,
    \begin{equation*}
        \EE_f(\| f - \hat{f_k}\|^2_{\omega_c}) \leq \| f - f_k\|^2_{\omega_c} + \frac{\sigma_c^2 k}{\pi n}
        \eqfinp
    \end{equation*}
By choosing $k = k_n$ such that $n^{-1}k_n \to 0$ and $k_n \to \infty$, $\hat{f_{n_k}}$ is a consistent estimator of $f$. 
\end{proposition}

If $X$ and $Y$ are two independant random variables, then 
    \begin{equation*}
        \Mcal[X \cdot Y](s) = \Mcal[X](s) \cdot \Mcal[Y](s)
        \eqfinv
        s \in \RR 
        \eqfinp
    \end{equation*}
\subsection{Statistical setting: estimation procedure}\label{s:model}
Consider the multiplicative compound process
\begin{equation}
    Y_t = \prod_{i = 1}^{N_t}X_i
    \eqfinv
    \label{multCompound}
\end{equation}
where $(N_t)_{t\geq0}$ is a Poisson process with constant intensity 
$\lambda > 0$, independent of the $\iid$ random variables $(X_j)_{j \in \NN}$ with common density $f \in L^1([0, \infty)) \cap L^2([0, \infty))$.\\ 

We suppose that we observe one trajectory of $(Y_t)$ over $[0,T]$ at equidistant time $\Delta, 2 \Delta, \ldots, n\Delta$, $T = n\Delta$.  
We denote $Z_{k\Delta} = \frac{Y_{k\Delta}}{Y_{{(k-1)}\Delta}}$ the $k-th$ increment of $Y$. \\

Let $c \in \Xi_X \cap \Xi_\Delta$. We have $1 \in \Xi_X \cap \Xi_\Delta$. Let
$\Mcal_1[f]$ be the Mellin transform of $X_1$ and $\Mcal_1[\Delta]$ be the Mellin transform of $Z_\Delta$. 

\begin{lemma}
\label{L:kintchine}
For any $s \in \CC$, \; $
    \Mcal_1[\Delta](s) = \exp\big(\Delta \lambda(\Mcal_1[f](s)-1)\big)
    \eqfinp$
\end{lemma}

From Lemma~\ref{L:kintchine}, we have
    \begin{equation*}
        \Mcal_1[f](s) = 1 + \frac{1}{\lambda \Delta} \log \Mcal_1[\Delta](s)
        \eqfinv
    \end{equation*}
where $ \log \Mcal_1[\Delta]$ is the distinguished logarithm of $\Mcal_1[\Delta]$ that we define in Appendix~\ref{S:DistingLog}. 

This last equation leads us to define the estimator
    \begin{equation*}
        \hat{\Mcal_1[f]}(s) =  1 + \frac{1}{\lambda \Delta} \hat{\log \Mcal_1[\Delta]}(s)
        \eqsepv
        s \in \RR
        \eqfinp
    \end{equation*}
where
    \begin{equation*}
        \hat{\Mcal_1[\Delta]}(t) = \frac{1}{n} \sum_{k = 1}^{n} Z_{k\Delta}^{c - 1 + it}
        \eqfinv  \quad
        \hat{\Mcal'_c[\Delta]}(t) = \frac{1}{n} \sum_{k = 1}^{n} i \log(Z_{k\Delta}) Z_{k\Delta}^{c - 1 + it} 
        \eqfinv
    \end{equation*}
    and 
    \begin{equation*}
        \hat{\log \Mcal_1[\Delta]}(s) = \int_0^u \frac{\hat{\Mcal'_c[\Delta]}(t)}{ \hat{\Mcal_1[\Delta]}(t)} \diff s
        \eqfinp
    \end{equation*}
However, the quantity $\hat{\log \Mcal_1[\Delta]}(s)$ may explode. Then we define 
    \begin{equation*}
        \tilde{\Mcal_1[f]}(s) =  \hat{\Mcal_1[f]}(s) \cdot \ind{| \hat{\Mcal_1[f]}(s)| \leq 4}
        \eqsepv
        s \in \RR  \
        \eqfinp
    \end{equation*}

Finally, we apply a Mellin inversion
    \begin{equation}
    \hat{f_{m, \Delta}}(x) = \frac{1}{2\pi} \int_{-m}^m x^{-c - it} \tilde{\Mcal_1[f]}(s) \diff t, 
    x \in (0, \infty)
    \eqfinv
    \end{equation}

\subsection{Risk bounds}
\label{sec:risk_mellin}
\begin{theorem} 
\label{T:risk}
Assume that $\EE[X_1^2] < \infty$, $\delta \leq \frac{1}{4} \log(n\Delta)$ and $n\Delta \to \infty$ as $J \to \infty$. Then, for any $m \geq 0$, it holds
\begin{equation}
    \EE \big[ \|\hat{f_{m, \Delta}}  - f \|_{\omega_1}^2\big] \leq \|f_m  - f \|_{\omega_1}^2 
    +
    \frac{1}{2\pi n \Delta^2} \int_{-m}^{m}\frac{1}{|\Mcal_1[\Delta](s)|^2}\diff s 
    + \frac{25}{\pi} \Big(\frac{\EE[\ln(X_1)^2]}{\Delta n} + 4 \frac{m}{(n \Delta)^{2}} \Big)
    \eqfinp
\end{equation}
\end{theorem} 
We prove this result in Appendix~\ref{S:ProofMain}. In the demonstration, we need \cref{lm:Mellin1} and \cref{lm:Mellin2}. 

\subsection{The adaptative procedure}
\label{sec:mellin_adaptive}
We want to minimize the right side term in Theorem~\ref{T:adapt}. The optimal cutoff $\bar{m_n}$ is defined by 
\begin{equation*}
    \bar{m_n} \in \argmin_{m > 0} \Big(\|f_m  - f \|_{\omega_1}^2 
    +
    \frac{1}{2\pi n \Delta^2} \int_{-m}^{m}\frac{1}{|\Mcal_1[\Delta](s)|^2}\diff s  \\
    + \frac{25}{\pi} \Big(\frac{\EE[\ln(X_1)^2]}{\Delta n} + 4 \frac{m}{(n \Delta)^{2}} \Big) \Big)
    \eqfinp
\end{equation*}

In the variance term, the leading term is in $\frac{m e^{4\Delta}}{n \Delta}$. Therefore, the optimal cutoff is such that
    \begin{equation*}
        |\Mcal_1[f](\bar{m_n})|^2 = \frac{e^{4\Delta}}{n \Delta}
        \eqfinp
    \end{equation*}

We can define the empirical optimal cutoff by 
    \begin{equation*}
        |\bar{\Mcal_1[f]}(\bar{m_n})|^2 = \frac{e^{4\Delta}}{n \Delta}
        \eqfinv
    \end{equation*}
where 
\begin{equation*}
    \bar{\Mcal_1[f]}(s) = \tilde{\Mcal_1[f]}(s) \ind{|\tilde{\Mcal_1[f]}(s)| \geq \kappa_{n, \Delta} / \sqrt{n \Delta}}
    \eqfinv
\end{equation*}
with 
\begin{equation*}
    \kappa_{n, \Delta} = (e^{2\Delta + \kappa}\sqrt{\log(n \Delta)}), \kappa > 0
    \eqfinp
\end{equation*}

We consider the empirical cutoff
\begin{equation*}
    \hat{m_n} = \max \big\{ m \geq 0:  |\bar{\Mcal_1[f]}(m)| = \kappa_{n, \Delta} / \sqrt{n \Delta} \big\} \wedge (n \Delta)^\alpha
    \eqfinp
\end{equation*}

\begin{theorem}[Adaptative method]
\label{T:adapt}
Assume that $\EE[X_1^2] < \infty$, $\delta \leq \frac{1}{4} \log(n\Delta)$ and $J\Delta \to \infty$ as $n \to \infty$. Then, 
there exist two positive constants $A$ and $B$ depending on $\kappa$ and $\EE[\ln(X_1)^2]$ such that
{\begin{multline*}
    \EE \big[ \|\bar{f_{\hat{m_n}, \Delta}}  - f \|_{\omega_1}^2\big] \\
    \leq \inf_{m\in [0,(n \Delta)^\alpha]}A \Big(\| f_m - f \|_{\omega_1}^2 + \frac{\log(n\Delta)m}{n\Delta}
    + \frac{1}{n \Delta^2} \int_{-m}^{m} \limits\frac{1}{|\Mcal_1[\Delta](s)|^2}\diff s 
    + 4 \frac{m^2}{(n \Delta)^{2}} \Big)
    \\
    + B\Big((n\Delta)^{\alpha - c(\Delta)^2} + \frac{1}{n\Delta}\Big)
    \eqfinp
\end{multline*}}
\end{theorem}

\section{Numerical experiments}
\label{S:simu}

\subsection{Decompounding with noise}

\subsubsection*{Implementation details}

For a given estimator $\hat f_{m, J}$ of $f$, we numerically evaluate the efficiency of $\hat f_{m, J}$. More specifically, we compute the $L^2$-risks $\espe \big( \norm{\hat f_{m, J} - f}^2 \big)$ by averaging the result obtained for 1000 independent simulations. \\

We test our method for three different jump distribution densities $f$: 
    \begin{itemize}
        \item the mixture density $0.3 \Ncal(-3.5, 1) + 0.7\Ncal(3.5, 1)$. 
        \item the Gamma density $\Gamma(2)$. 
        \item the Cauchy density $\Ccal(0, 1)$.
    \end{itemize}
and the noise density $\Ncal(0, 1)$.

\subsubsection*{Efficiency of the estimation} The estimator is expected to be progressively more accurate as J increases. We compute the estimator for different values of $J$ for mixed density (Figure \ref{fig:mixed}) and Gamma density (Figure \ref{fig:gamma}).
\begin{figure}
    \centering
    \includegraphics[scale=0.4]{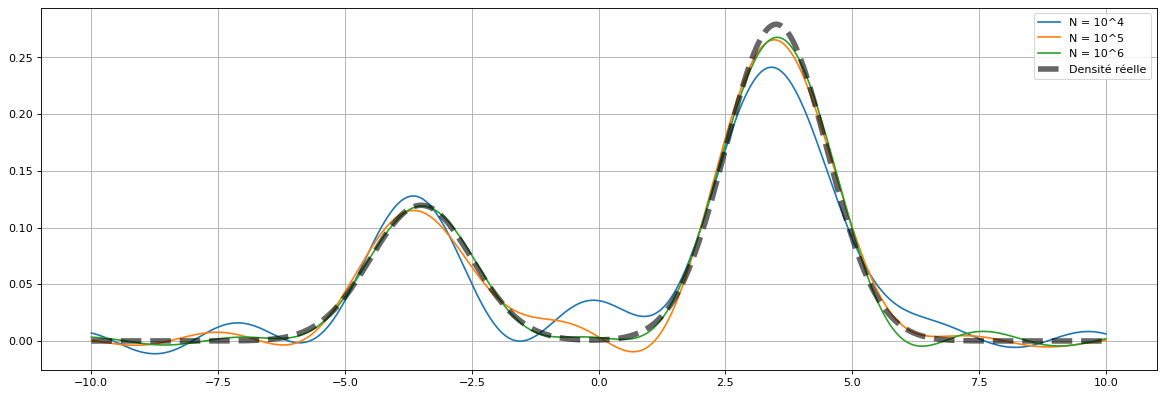}
    \caption{Illustration of the estimator. \\
    \textit{Input: Jump distribution $X \sim 0.3 \Ncal(-3.5, 1) + 0.7\Ncal(3.5, 1)$. Observations are corrupted by a Gaussian noise $\Ncal(0, 1)$. $t_1 = 0.5, t_2 = 1, J=10^5.$}}
    \label{fig:mixed}
\end{figure}
\begin{figure}
    \centering
    \includegraphics[scale=0.4]{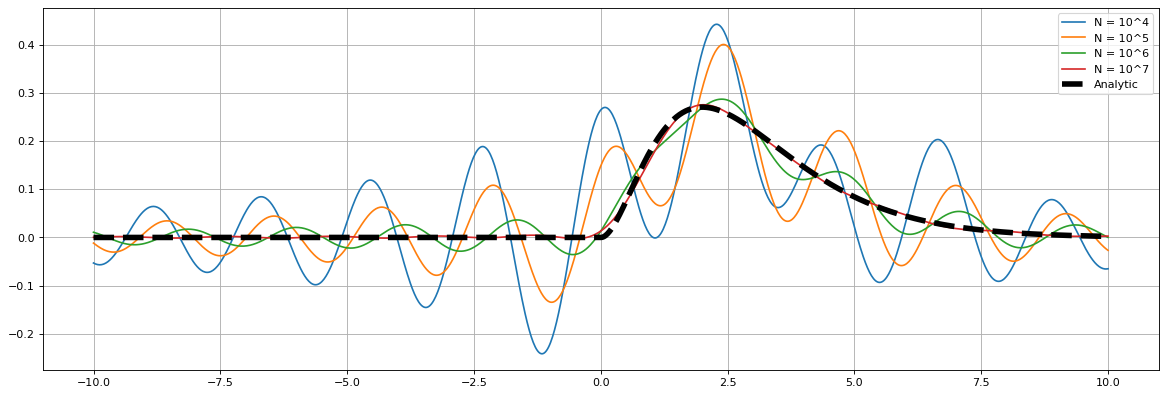}
    \caption{Illustration of the estimator. \\
    \textit{Input: Jump distribution $X \sim \Gamma(3)$. Observations are corrupted by a Gaussian noise $\Ncal(0, 1)$. $t_1 = 0.5, t_2 = 1, J=10^5.$}}
    \label{fig:gamma}
\end{figure}

\subsubsection*{Dependence of the $L^2$-risk to the cut-off $m$ and efficiency of the adaptive procedure}
\label{S:simuM}

The value of $m$ has a strong impact on the quality of the estimator $\hat f_{m, J}$. When $m$ is too low or too high, we expect that the estimator should deteriorate. The adaptive procedure gives an automatic way to select a "good" $m$ based on observations. \\

We illustrate the efficiency of the adaptive procedure for the mixture density (Figure \ref{fig:mopt_mixed}), Gamma density (Figure \ref{fig:mopt_gamma}) and the Cauchy density (Figure \ref{fig:mopt_cauchy}).

The results of the simulations are satisfactory, as we can see that the value of $m$ chosen seems to minimize our error.

\begin{figure}
    \centering
    \includegraphics[scale=0.6]{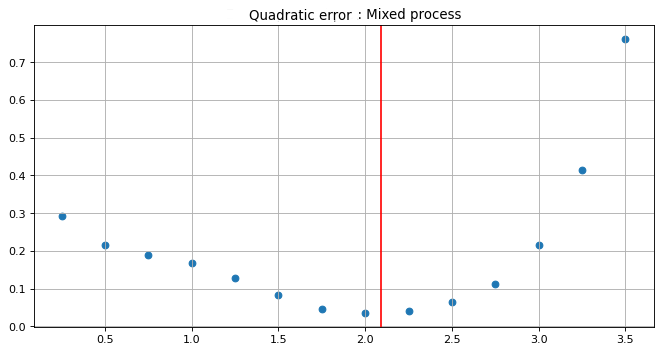}
    \caption{Computations of the $L^2$-risks $\espe \big( \norm{\hat f_{m, J} - f}^2 \big)$ (y axis) for different values of $m$ (x axis). The vertical line represents the value of $\hat m$ chosen automatically by the adaptive procedure. \\
    \textit{Input: Jump distribution $X \sim 0.3 \Ncal(-2, 1) + 0.7\Ncal(2, 1)$. Observations are corrupted by a Gaussian noise $\Ncal(0, 1)$. $t_1 = 0.5, t_2 = 1, J=10^5.$}}
    \label{fig:mopt_mixed}
\end{figure}
\begin{figure}
    \centering
    \includegraphics[scale=0.6]{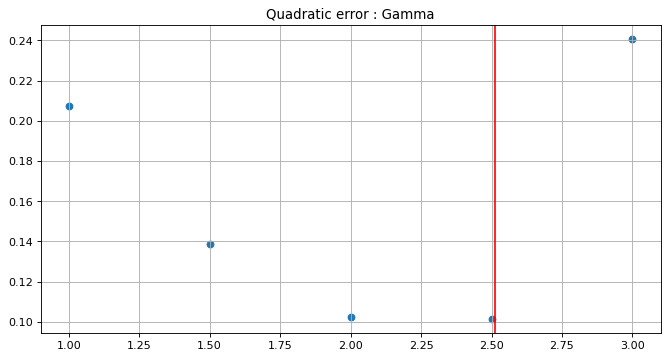}
    \caption{Computations of the $L^2$-risks $\espe \big( \norm{\hat f_{m, J} - f}^2 \big)$ (y axis) for different values of $m$ (x axis). The vertical line represents the value of $\hat m$ chosen automatically by the adaptive procedure. \\
    \textit{Input: Jump distribution $X \sim \Gamma(2)$. Observations are corrupted by a Gaussian noise $\Ncal(0, 1)$. $t_1 = 0.5, t_2 = 1, J=10^5.$}}
    \label{fig:mopt_gamma}
\end{figure}
\begin{figure}
    \centering
    \includegraphics[scale=0.6]{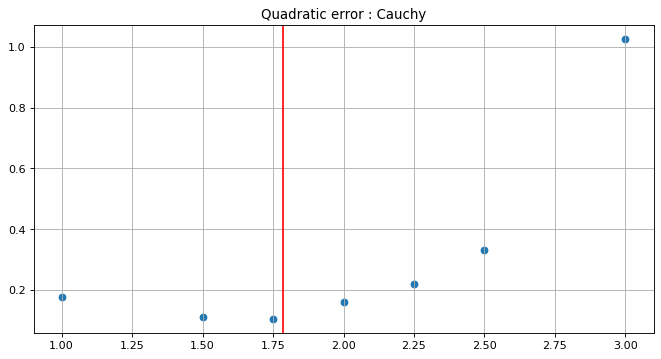}
    \caption{Computations of the $L^2$-risks $\espe \big( \norm{\hat f_{m, J} - f}^2 \big)$ (y axis) for different values of $m$ (x axis). The vertical line represents the value of $\hat m$ chosen automatically by the adaptive procedure. \\
    \textit{Input: Jump distribution $X \sim \Ccal(0, 1)$. Observations are corrupted by a Gaussian noise $\Ncal(0, 1)$. $t_1 = 0.5, t_2 = 1, J=10^5.$}}
    \label{fig:mopt_cauchy}
\end{figure}

\newpage
\subsubsection*{How to select an optimal time $t_2$}
\label{S:simuT}
To define our estimator, we have chosen two times $t_1$ and $t_2$ in an arbitrary way. Therefore, we would like to know how to select them in an optimal way.\\

Theorem \ref{thm:analysis_error} states that $t_2$ must be neither too large nor too close to $t_1$, which we can illustrate with numerical simulation. Figure \ref{fig:error_t2} shows the value of the $L^2$-risks $\espe \big( \norm{\hat f_{m, J} - f}^2 \big)$ of the estimator for different values of $t_2$.
Each of these values is the average of the results of 1000 different simulations. Here, it seems that the estimate is better when $t_2$ is between $1$ and $1.5$, which seems to correspond roughly to the intensity of the number of jumps.

\begin{figure}[h!]
    \centering
    \includegraphics[width=0.6\paperwidth]{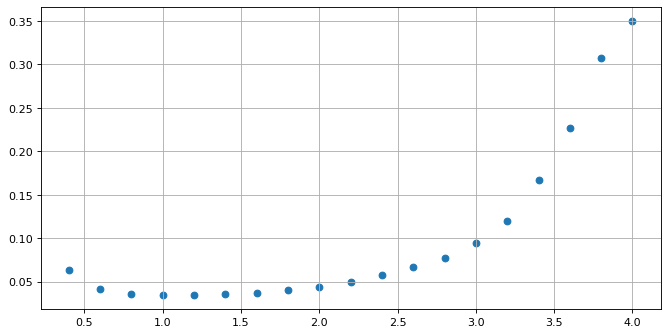}
    \caption{Computations of the $L^2$-risks $\espe \big( \norm{\hat f_{m, J} - f}^2 \big)$ (y axis) for different values of $t_2
    $ (x axis). \\
    \textit{Input: Jump distribution $X \sim 0.3 \Ncal(-2, 1) + 0.7\Ncal(2, 1)$. Observations are corrupted by a Gaussian noise $\Ncal(0, 1)$. $t_1 = 0.2$, $m = 2$, $J = 10^5$. }}
    \label{fig:error_t2}
\end{figure}
\subsection{Multiplicative decompounding}

In this section, we illustrate the efficiency of the Mellin estimator we built in \cref{sec:mellin_estimator} on a simulated sample of size $n = 5000$. 

Here, we assume that the jump distribution densities $f$ is the Beta law $\beta(200, 30)$. The bandwidth $m = 83.7$ was selected by the adaptive procedure described in \cref{sec:mellin_adaptive}. The resulting estimate is plotted in \cref{fig:mellin.beta5000}.

\begin{figure}[H]
    \centering
    \includegraphics[width=0.55\paperwidth]{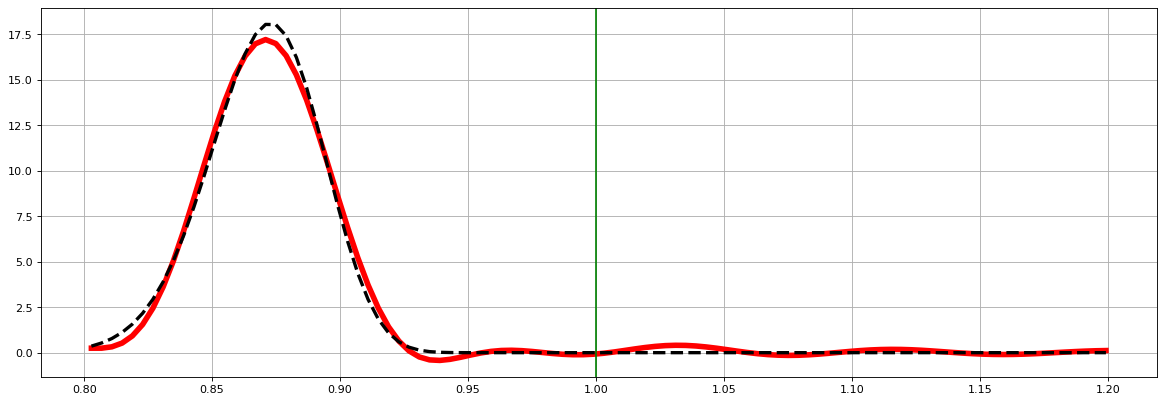}
    \includegraphics[width=0.55\paperwidth]{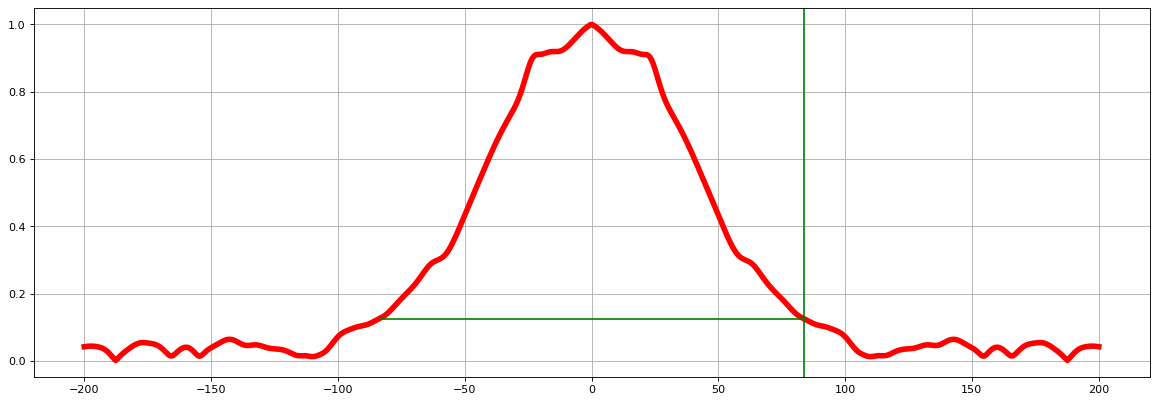}
    \caption{\textbf{Top: }Estimation of the Beta density $\beta(200, 30)$ for 5000 observations. The real density is the black dotted-line and the estimator is plotted in red. \textbf{Bottom: }Illustration of the adaptive procedure. The green lines show the cutoff on the Mellin transform.}\label{fig:mellin.beta5000}.
  \end{figure}

\section{Proofs}
\label{sec:proof}
\subsection{Proof of Theorem \ref{thm:analysis_error}}

This proof is an adaptation of the proof of Theorem 3.1 of Duval and Kappus \cite{duval2019adaptive}. Nevertheless, we have in addition a noise term that must be taken into account in our majorations.\\

From the triangle inequality and Parseval's equality we have 
    \begin{subequations}
        \begin{align*}
            \norm{\hat f_{m, J} - f}^2
            &\leq \norm{f_m - f}^2 + \norm{\hat f_{m, J} - f_m}^2 \\
            &= \norm{f_m - f}^2 + \frac{1}{2\pi} \int_{-m}^m | \tilde \varphi_X^J (u) - \varphi_X(u)|^2 \du
            \eqfinp
        \end{align*}
    \end{subequations}

By recalling that for any reals $(a, b) \in \RR^2$, $|a+b|^2 \leq 2 (a^2 + b^2)$, 
and by using Definition \ref{def:estimateur} and Definition \ref{def:estimateur_fctcar_tilde}, then we have
    \begin{subequations}
    \begin{align*}\l
         &\int_{-m}^m |\tilde \varphi_X^J  - \varphi_X(u)|^2 \du \notag{}\\
         &= \frac{1}{(t_2 - t_1)^2}\int\limits_{-m}^m 
        | \log \tilde \varphi_{Z_{t_2}}^{\, J}  (u) 
         - \log \tilde \varphi_{Z_{t_1}}^{\, J}  (u) 
         - \log\varphi_{Z_{t_2}} (u) + \log \varphi_{Z_{t_1}} (u)|^2 \du \\
         &\leq \frac{2}{(t_2 - t_1)^2}\int\limits_{-m}^m |\log \tilde \varphi_{Z_{t_1}}^{\, J} (u) - \log\varphi_{Z_{t_1}} (u)|^2 \du + \frac{2}{(t_2 - t_1)^2}\int\limits_{-m}^m | \log \tilde \varphi_{Z_{t_2}}^{\, J}  (u)  - \log\varphi_{Z_{t_2}} (u)|^2 \du 
\eqfinp
    \end{align*}
    \end{subequations}
    
In the rest of the proof we establish a majorization of the two terms on the right. \\

Let $\tau \in \{t_1, t_2\}$ and $I_{\tau, J, m}$ denote the quantity 
	\begin{equation*}
	I_{\tau, J, m} = \frac{1}{(t_2 - t_1)^2} \int\limits_{-m}^m | \log \tilde \varphi_{Z_{\tau}}^{\, J} (u) - \log\varphi_{Z_{\tau}} (u)|^2 \du
	\eqfinp
	\end{equation*}

Let $(\gamma, \zeta) \in (0, \infty)^2$ be two positive reals such that $\gamma c_m > \zeta$. \\
In particular,  we take $\gamma$ such that $\gamma c_m = \frac{2 \zeta}{1 \wedge (t_2 - t_1)} > \zeta$. \\

Consider the events 
    \begin{equation*}
        \Omega_{\zeta, \tau}(m) = \Bigg\{ 
            \forall u \in [-m, m], |\hat \varphi^J_{Z_{\tau}}(u) - \fctcar{Z_{\tau}}(u)|
            \leq
            \zeta \sqrt{\frac{\log(J \tau)}{J \tau}}
        \Bigg\}
    \end{equation*}
and
    \begin{equation*}
        M_{J, t}^{(\gamma c_m)} = \min \Big\{ u \geq 0: |\varphi_{Z_t}(u)| = \gamma c_m \sqrt{\log(Jt)/(Jt)} \Big\}
        \eqfinv
    \end{equation*}
as defined in Lemma \ref{lm:main1} and Lemma \ref{lm:main2}. \\
As the events
	\begin{equation*}
    \{ |u| \leq m \wedge M_{J, \tau}^{(\gamma c_m)} \} \cap \Omega_{\zeta, \tau}(m)
    \eqsepv
    \{ |u| \in [M_{J, \tau}^{(\gamma c_m)}, m] \} \cap \Omega_{\zeta, \tau}(m)
    \mtext{ and }
    \Omega_{\zeta, \tau}(m)^c
	\end{equation*}
form a partition of the sample set $\Omega$, we only need to control $I_{\tau, J, m}$ on each of them. This leads us to define the quantities 
	\begin{subequations}
	\begin{align*}
		I_{\tau, J, m}^{(1)} 
		&= \frac{1}{(t_2 - t_1)^2}
		\int\limits_{-m \wedge M_{J, \tau}^{(\gamma c_m)}}^{m \wedge M_{J, \tau}^{(\gamma c_m)}} \1{\Omega_{\zeta, \tau}(m)}
		|\log \tilde \varphi_{Z_{\tau}}^{\, J}  (u) - \log\varphi_{Z_{\tau}} (u)|^2 \du
		\eqfinv
	    \\
		I_{\tau, J, m}^{(2)}
		&= \frac{1}{(t_2 - t_1)^2}
		\1{m > M_{J, \tau}^{(\gamma c_m)}} \cdot \1{\Omega_{\zeta, \tau}(m)} \int_{|u| \in [M_{J, \tau}^{(\gamma c_m)}, m]} |\log \tilde \varphi_{Z_{\tau}}^{\, J}  (u) - \log\varphi_{Z_{\tau}} (u)|^2 \du 
		\eqfinv
		\\
		I_{\tau, J, m}^{(3)} 
		&= \frac{1}{(t_2 - t_1)^2}
		\1{\Omega_{\zeta, \tau}(m)^c}  \int\limits_{-m}^m | \log \tilde \varphi_{Z_{\tau}}^{\, J}  (u) - \log\varphi_{Z_{\tau}} (u)|^2 \du
		\eqfinp
	\end{align*}
	\end{subequations}

Obviously, we have the decomposition 
	\begin{equation*}
	I_{\tau, J, m} = I_{\tau, J, m}^{(1)} + I_{\tau, J, m}^{(2)}  + I_{\tau, J, m}^{(3)} 
	\eqfinp
	\end{equation*}
	
\paragraph{\textbf{Step 1} -- Control of $I_{\tau, J, m}^{(1)}$} Define the event $\Acal = \{ |u| \leq m \wedge M_{J, \tau}^{(\gamma c_m)} \}  \cap \Omega_{\zeta, \tau}(m)$. The triangle inequality ensures that
	\begin{equation*}
	| \log \hat \varphi_{Z_\tau}^{\, J}(u) | \leq | \log \hat \varphi_{Z_\tau}^{\, J}(u) - \log \varphi_{Z_\tau}(u) | + |\log \varphi_{Z_\tau}(u) |
	\eqfinp
	\end{equation*}
Then Lemma \ref{lm:main2} applied to $\zeta < \gamma c_m$ ensures that
	\begin{equation*}
	| \log \hat \varphi_{Z_\tau}^{\, J}(u) | \leq 
	        \frac{\gamma c_m}{\zeta} \log \Big( \frac{\gamma c_m}{\gamma c_m- \zeta}\Big)
        \frac{|\hat \varphi_{Z_\tau}^J(u) - \varphi_{Z_\tau}(u)|}{|\varphi_{Z_\tau}(u)|} + |\log \varphi_{Z_\tau}(u) |
	\eqfinp
	\end{equation*}
Then we have by definition of $\Omega_{\zeta, \tau}(m)$, 
		\begin{equation*}
	| \log \hat \varphi_{Z_\tau}^{\, J}(u) | \leq 
	        \frac{\gamma c_m}{\zeta} \log \Big( \frac{\gamma c_m}{\gamma c_m- \zeta}\Big)
        \frac{1}{|\varphi_{Z_\tau}(u)|} \zeta \sqrt{\frac{\log(J\tau)}{J \tau}}+ |\log \varphi_{Z_\tau}(u) |
	\eqfinp
	\end{equation*}
For $J\tau$ large enough, we have $\gamma c_m \sqrt{\frac{\log(J \tau)}{J \tau}} < 1$, hence on $A$
the definition of $M_{J, \tau}^{(\gamma c_m)}$ ensures that
	\begin{equation*}
	\frac{1}{|\varphi_{Z_\tau}(u)|} \leq \frac{1}{\gamma c_m \sqrt{\frac{\log(J\tau)}{J \tau}}}
	\eqfinp
	\end{equation*}
	
It follows that 
		\begin{subequations}
		\begin{align*}
			| \log \hat \varphi_{Z_\tau}^{\, J}(u) |
			&\leq
	        \frac{\gamma c_m}{\zeta}
	        \log \Big( \frac{\gamma c_m}{\gamma c_m- \zeta}\Big)
	        \frac{1}{\gamma c_m \sqrt{\frac{\log(J\tau)}{J \tau}}} 
	        \zeta \sqrt{\frac{\log(J\tau)}{J \tau}}
	        + 
	        |\log \varphi_{Z_\tau}(u) |
	        \eqfinv
	        \\
	        &\leq
	        \log \Big( \frac{\gamma c_m}{\gamma c_m- \zeta}\Big)
	        + 
	        |\log \varphi_{Z_\tau}(u) |
	        \eqfinp
		\end{align*}
	\end{subequations}
By definition, this means (see also Proposition \ref{duval2017nonparametric_lm3})
		\begin{subequations}
		\begin{align*}
			| \log \hat \varphi_{Z_\tau}^{\, J}(u) |
	        &\leq
	        	        \log \Big( \frac{\gamma c_m}{\gamma c_m- \zeta}\Big)
	        + 
	        |\log \varphi_{Y_\tau}(u) + \log \varphi_{\varepsilon}(u)  |
	        \eqfinv
		\end{align*}
	\end{subequations}

However, we know that for all $u \in \RR$
	\begin{equation*}
		\varphi_X(u) = 1 + \frac{1}{\tau} \log \varphi_{Y_\tau} (u)
		\eqfinp
	\end{equation*}
	
It follows that 
	\begin{subequations}
		\begin{align*}
			| \log \hat \varphi_{Z_\tau}^{\, J}(u) |
	        &\leq
	        \log \Big( \frac{\gamma c_m}{\gamma c_m- \zeta}\Big)
	        + 
	        | \tau (\varphi_X(u) + 1) + \log \varphi_{\varepsilon}(u)  |
	        \\
	        & \leq
			\log \Big( \frac{\gamma c_m}{\gamma c_m- \zeta}\Big)
	        + 
	        2 \tau + | \log \varphi_{\varepsilon}(u)  |
	        \eqfinv
	        \\
	        & \leq
			\log \Big( \frac{\gamma c_m}{\gamma c_m- \zeta}\Big)
	        + 
	        2 \tau + \sup_{[-m, m]}{| \log \varphi_{\varepsilon}(\cdot)|}
	        \eqfinp
		\end{align*}
	\end{subequations}

Therefore,
		\begin{align*}
			\bigg| \frac{\log \hat \varphi_{Z_\tau}^{\, J}(u)}{t_2 - t_1} \bigg|
	        &\leq \frac{1}{t_2 - t_1}
			\log \Big( \frac{\gamma c_m}{\gamma c_m- \zeta}\Big)
	        + 
	        \frac{2 \tau}{t_2 - t_1} + \frac{\sup_{[-m, m]}{| \log \varphi_{\varepsilon}(\cdot)|}}{t_2 - t_1}
         \eqfinp
		\end{align*}

There are two different cases. 
\begin{itemize}
\item If $t_2 - t_1 \geq 1$, then $\gamma c_m = \frac{2 \zeta}{1 \wedge (t_2 - t_1)} = 2 \zeta$.
It follows that
\begin{equation*}
\frac{1}{t_2 - t_1}
\log \Big( \frac{\gamma c_m}{\gamma c_m- \zeta}\Big)
= 
\frac{1}{t_2 - t_1} \log(2) < 1
\eqfinp
\end{equation*}
\item If $t_2 - t_1 < 1$, then $\gamma c_m = \frac{2 \zeta}{1 \wedge (t_2 - t_1)} = \frac{2 \zeta}{t_2 - t_1} $. 
It follows that 
	\begin{subequations}
	\begin{align*}
	 \frac{1}{t_2 - t_1}
	 \log \Big( \frac{\gamma c_m}{\gamma c_m- \zeta}\Big)
	 &= 
	 \frac{1}{t_2 - t_1} \log \Bigg( \frac{\frac{2 \zeta}{t_2 - t_1}}{\frac{2 \zeta}{t_2 - t_1} - \frac{\zeta (t_2 - t_1)}{t_2 - t_1}  } \Bigg) \\
	 &= \frac{1}{t_2 - t_1} \log \Bigg( \frac{2}{2 - (t_2 - t_1)} \Bigg)\\
	 &\leq 1
	 \eqfinv
	 \end{align*}
	\end{subequations}
because for any real $x \in (0, 1)$ we have
	\begin{equation*}
	0 < \frac{1}{x}\log\Big(\frac{2}{2-x}\Big) < 1
	\eqfinp
	\end{equation*}
\end{itemize}

It leads to 
\begin{align}
	\bigg| \frac{\log \hat \varphi_{Z_\tau}^{\, J}(u)}{t_2 - t_1} \bigg|
	&\leq 1 + \frac{2 \tau}{t_2 - t_1} + \frac{\sup_{[-m, m]}{| \log \varphi_{\varepsilon}(\cdot)|}}{t_2 - t_1} \notag{}\\
    &\leq 1 + \frac{2 t_2}{t_2 - t_1} + \frac{\sup_{[-m, m]}{| \log \varphi_{\varepsilon}(\cdot)|}}{t_2 - t_1}.
	        \label{eq:tilde_egal_hat}
\end{align}

Since we consider $0 \leq m \leq C_{t_1, t_2}^J$, we have
	\begin{equation*}
	\bigg|\log \hat \varphi_{Z_\tau}^{\, J}(u)\bigg| 
	\leq 3t_2 - t_1 + \sup_{[-m, m]}{| \log \varphi_{\varepsilon}(\cdot)|}
	\leq \ln(J)
    \eqfinp
	\end{equation*}
Therefore, it comes from Equation~\eqref{eq:tilde_egal_hat} and Equation~\eqref{def:estimateur_fctcar_tilde} that 
	\begin{equation*}
	\log \tilde \varphi_{Z_{\tau}}^{\, J}  (u)  = \log \hat \varphi_{Z_{\tau}}^{\,J} (u)
	\eqfinp
	\end{equation*}

By using Lemma~\ref{lm:main2} and the definition of $\gamma$, we have
    \begin{subequations}
    \begin{align*}
        \espe[I_{\tau, J, m}^{(1)}]
		&= \frac{1}{(t_2 - t_1)^2}
		\int\limits_{-m \wedge M_{J, \tau}^{(\gamma c_m)}}^{m \wedge M_{J, \tau}^{(\gamma c_m)}} \E \Big[  \1{\Omega_{\zeta, \tau}(m)}
		|\log \hat \varphi_{Z_{\tau}}^{\, J}  (u) - \log\varphi_{Z_{\tau}} (u)|^2 \Big] \du \\
        &\leq \frac{1}{(t_2 - t_1)^2}  \int_{-m}^{m} 
        \frac{\espe[|\hat \varphi_{Z_\tau}^J(u) - \varphi_{Z_{\tau}}(u)|^2]}{|\varphi_{Z_{\tau}}(u)|^2} \du
        \eqfinv
    \end{align*}
    \end{subequations}

because $\frac{\gamma c_m}{\zeta} \log \big(\frac{\gamma c_m}{\gamma c_m - \zeta} \big) \leq 1$ due to the fact that $\gamma c_m > \zeta$.

Direct computations lead to
    \begin{subequations}
    \begin{align*}
       \espe[|\hat \varphi_{Z_\tau}^J(u) - \varphi_{Z_\tau}(u)|^2]
        &= \frac{1 - |\varphi_{Z_{\tau}}(u)|^2}{J} \\
        &= \frac{1 - |\varphi_{Y_{\tau}}(u)|^2 |\varphi_{\varepsilon}(u)|^2}{J} \\
        &= \frac{(1 - |\varphi_{Y_{\tau}}(u)|^2) |\varphi_{\varepsilon}(u)|^2 + 1 - |\varphi_{\varepsilon}(u)|^2 }{J} \\
        &\leq \frac{(1 - |\varphi_{Y_{\tau}}(u)|^2) |\varphi_{\varepsilon}(u)|^2 + 1 - |\varphi_{\varepsilon}(u)|^2 }{J} \\
        &\leq \frac{2}{J}
        \eqfinp
    \end{align*}
    \end{subequations}

We deduce that
    \begin{equation*}
        \espe[I_{\tau, J, m}^{(1)}]
        \leq 
        \frac{2}{J (t_2 - t_1)^2}  \int_{-m}^{m} 
        \frac{1}{|\varphi_{Z_{\tau}}(u)|^2} \du
        \eqfinp
    \end{equation*}
    
Futhermore, we know that $(Y_\tau)$ is a compound Poisson process with intensity $\lambda = 1$. It leads to 
	\begin{equation*}
	|\varphi_{Y_\tau}(u)| \geq e^{-2\tau} 
	\eqsepv
	u \in \RR
	\eqfinp
	\end{equation*}

Then, 
	\begin{equation*}
	| \varphi_{Z_{\tau}}(u)|^2 = |\varphi_{Y_\tau}(u) \varphi_{\varepsilon}(u)|^2 \geq e^{-4\tau} |\varphi_{\varepsilon}(u)|^2
 \eqfinp
	\end{equation*}

It leads to 
    \begin{equation*}
        \espe[I_{\tau, J, m}^{(1)}]
        \leq 
        \frac{2 e^{4\tau}}{J (t_2 - t_1)^2}  \int_{-m}^{m} 
        \frac{1}{|\varphi_{\varepsilon}(u)|^2} \du
        \eqfinp
    \end{equation*}

\paragraph{\textbf{Step 2} -- Control of $I_{\tau, J, m}^{(3)}$}

    on the event $\Omega_{\zeta, \tau}(m)^c$, it is more complicated to control the empirical distinguished logarithm. 
    Nevertheless, the cut-off on the estimator allows us to write that

\begin{align*}
	 \int\limits_{-m}^m | \log \tilde \varphi_{Z_{\tau}}^{\, J}  (u) - \log\varphi_{Z_{\tau}} (u)|^2 \du 
	 &\leq 
	 2  \int\limits_{-m}^m | \log \tilde \varphi_{Z_{\tau}}^{\, J}  (u)|^2 \du 
	 + 
	 2 \int\limits_{-m}^m | \log\varphi_{Z_{\tau}} (u)|^2 \du
	 \\
	 &\leq 
	 4 m \rj
	 + 
	 2 \int\limits_{-m}^m | \log\varphi_{Y_{\tau}}(u) + \log\varphi_\varepsilon (u)|^2 \du \\
	 &\leq 
	 4 m \rj
	 + 
	 4 \int\limits_{-m}^m | \log\varphi_{Y_{\tau}}(u)|^2 \du 
	 +
	 4 \int\limits_{-m}^m | \log\varphi_\varepsilon (u)|^2 \du\\
	 &\leq 
	 4 m \rj
	 + 
	 32 m \tau^2 \du 
	 +
	 4 \int\limits_{-m}^m | \log\varphi_\varepsilon (u)|^2
	 \eqfinp
\end{align*}

Indeed, we have that $\varphi_X = 1 + \frac{1}{\tau}  \log\varphi_{Y_{\tau}}
	$. It follows that $|\log\varphi_{Y_{\tau}}| \leq 2 \tau$.

Recall that we have 
	\begin{equation*}
		I_{\tau, J, m}^{(3)} 
		= \frac{1}{(t_2 - t_1)^2}
		\1{\Omega_{\zeta, \tau}(m)^c}  \int\limits_{-m}^m | \log \tilde \varphi_{Z_{\tau}}^{\, J}  (u) - \log\varphi_{Z_{\tau}} (u)|^2 \du
	\eqfinv
	\end{equation*}
which leads to write
    \begin{equation*}
        \E[I_{\tau, J, m}^{(3)}] 
        \leq 
     \frac{1}{(t_2 - t_1)^2} \cdot \Big( 
	 4 m \rj
	 + 
	 32 m \tau^2 \du 
	 +
	 4 \int\limits_{-m}^m | \log\varphi_\varepsilon (u)|^2 \du \Big)
	 \cdot 
	 \PP (\Omega_{\zeta, \tau}(m)^c)
	 \eqfinp
	\end{equation*}

We apply Lemma~\ref{lm:main1} with $\eta = 2$ and $\zeta > \sqrt{5 \tau}$. It ensures that

    \begin{equation*}
        \E[I_{\tau, J, m}^{(3)}] 
        \leq 
        \frac{1}{(t_2 - t_1)^2} \cdot \Big( 
	 4 m \rj
	 + 
	 32 m \tau^2
	 +
	 4 \int\limits_{-m}^m | \log\varphi_\varepsilon (u)|^2 \du \Big)
	 \cdot 
      \bigg( \frac{\E[X_i^2]}{J \tau} 
			+ \frac{\E[\noise^2]}{J \tau^2} + 4 \frac{m}{(J \tau)^2} \bigg)
   \eqfinp
	\end{equation*}

\paragraph{\textbf{Step 3} -- Control of $I_{\tau, J, m}^{(2)}$}
By Assumption, $\tau \leq \delta \log(J \tau)$ and $\delta < 1/4$.
It follows that for all $u \in [-m, m], $,
    \begin{equation*}
     |\varphi_{Z_{\tau}}(u)| = |\varphi_{Y_{\tau}} |\cdot |\varphi_\varepsilon| > e^{-2\tau} c_m
	 >
  (J \tau)^{-2\delta} c_m
  >
	 \gamma c_m \sqrt{\log(J\tau)/(J\tau)} \Big\},
 	\end{equation*}
when $J\tau$ is enough high.\\
	
Finally, $M_{J, \tau}^{(\gamma c_m)} \geq m$, and $I_{\tau, J, m}^{(2)} = 0 $ almost surely. 

\subsection{Proof of Theorem \ref{thm:oracle}}

Let $m \geq 0$. The proof of the theorem is divided in two steps: 
in a first one, we control $\EE \norm{ \hat f_{\hat m_n, J} - f}^2 $ on the event $\Ecal = \{ \hat m_J < m \}$, then we control it on the complementary event $\Ecal^c = \{ \hat m_J \geq m \}$. \\

As we know, it follows from the triangle inequality and Parseval's equality that 
    \begin{align*}
        \norm{\bar f_{m, J} - f}^2
        &= \norm{f_m - f}^2 + \frac{1}{2\pi} \int_{-m}^m | \tilde \varphi_X^J (u) - \varphi_X(u)|^2 \du \\
        &= \frac{1}{2\pi}  \int\limits_{|u| \in [-m, m]^c} \hspace{-15pt} |\varphi_X(u)|^2 \du  + \frac{1}{2\pi} \int_{-m}^m | \tilde \varphi_X^J (u) - \varphi_X(u)|^2 \du  
        \eqfinp
    \end{align*}

Therefore, by replacing the estimator $\hat f_{m, J}$ by $\hat f_{\hat m_n, J}$ produces a surplus of bias on $\Ecal$ and and surplus of variance on $\Ecal^c$. Then, we only need to control these surpluses.
\paragraph{\textbf{Step 1} -- Control on $\Ecal$}

\begin{align*}
    \EE \Big[ \1{\Ecal} \int\limits_{|u| \in [\hat m_J, m]}\hspace{-15pt} |\varphi_X &(u)|^2 \du \Big]
    \leq 2 \EE \Big[ \1{\Ecal} \int\limits_{|u| \in [\hat m_J, m]}\hspace{-15pt} |\tilde \varphi_X^J(u)|^2 \du \Big] + 2 \EE \Big[ \1{\Ecal} \int\limits_{|u| \in [\hat m_J, m]}\hspace{-15pt} |\tilde \varphi_X^J(u) - \varphi_X(u)|^2 \du \Big] \\
    &\leq  2 \EE \Big[ \1{\Ecal} \int\limits_{|u| \in [\hat m_J, m]} \frac{\kappa_{J, t_1, t_2}^2 \cdot (1 + u^2)}{J ( t_2 - t_1)^2} \du \Big] +  2 \EE \Big[ \1{\Ecal} \int\limits_{|u| \in [0, m]}\hspace{-15pt} |\tilde \varphi_X^J(u) - \varphi_X(u)|^2 \du \Big] \\
    &\leq 4m \cdot(1 + m^2)\cdot \frac{\kappa_{J, t_1, t_2}^2}{J ( t_2 - t_1)^2} + 2A
    \eqfinv
\end{align*}
where $A = \sum_{i = 1}^2 \frac{4 e^{4t_i}}{J (t_2 - t_1)^2} \int_{-m}^{m} \frac{\du}{|\varphi_{\varepsilon}(u)|^2}
		+\frac{4 K_{J, t_1, t_2}}{(t_2 - t_1)^2} 
	 \cdot \bigg( \frac{\E[X_i^2]}{J t_i} 
			+ \frac{\E[\noise^2]}{J t_i^2} + 4 \frac{m}{(J t_i)^2} \bigg)$
is given by Theorem \ref{thm:analysis_error}. \\
    
Given that $ \kappa_{J, t_1, t_2} = 2 \sqrt{2} e^{2 t_2} + \kappa \sqrt{\ln(J ( t_2 - t_1)^2)} $, we have that 
    \begin{equation*}
     \EE \big[ \norm{\bar f_{\hat m_J} - f }^2 \1{\Ecal} \big] 
     \leq 
     \norm{f_m - f}^2 + C \frac{m \cdot(1 + m^2) \cdot \ln(J ( t_2 - t_1)^2) }{J ( t_2 - t_1)^2} + \tilde C A
     \eqfinv
    \end{equation*}

where $C$ and $\tilde C$ are two constantes.
    
\paragraph{\textbf{Step 2} -- Control on $\Ecal^c$}

Now, we look at the case where where $\hat m_J > m$. By construction, $\hat m_J$ is bounded by $(\mmax)^\alpha$. It follows that 
    \begin{equation*}
        \EE \Big[ \int\limits_{m < |u| < \hat m_J }\hspace{-15pt} |\bar \varphi_X^J(u) - \varphi(u)|^2 \du \1{\Ecal^c} \Big] 
        \leq
        \EE \Big[ \int\limits_{m < |u| < (\mmax)^\alpha }\hspace{-30pt} |\bar \varphi_X^J(u) - \varphi(u)|^2 \du \Big]
        \eqfinp
    \end{equation*}

Let $\eta > 2$, such that $\alpha - \eta < -1$.
As for the proof of the theorem \ref{thm:analysis_error}, Lemma \ref{lm:main1} and Lemma \ref{lm:main2} leads to 
    \begin{align*}
         \EE \big[ |\bar \varphi_X^J(u) - \varphi(u)|^2 \big]
         &
        \leq |\varphi_X(u)|^2 +
        \EE \big[ |\hat \varphi_X^J(u) - \varphi(u)|^2 \big]
         \\
         &
         \leq |\varphi_X(u)|^2 
         +
        \sum_{i = 1}^2 \frac{4 e^{4t_i}}{J (t_2 - t_1)^2} \frac{\du}{|\varphi_{\varepsilon}(u)|^2}
         + 
         \frac{\espe[X_1^2]}{J t_1} + \frac{ \espe[\noise^2]}{J t_1^2} + 4 \frac{m}{(J t_1)^\eta} 
         \eqfinp
    \end{align*}   
On the event $$  \Big|\frac{ \varphi_X(u)}{\sqrt{1 + u^2}}  \Big| \geq \frac{2e^{2 t_2}}{\sqrt{J}(t_2 - t_1)}\eqfinv$$ we see that 
        \begin{align*}
         \EE \big[ |\bar \varphi_X^J(u) - \varphi(u)|^2 \big]
         &\leq|\varphi_X(u)|^2 + \frac{8}{J(t_2 - t_1)^2}\frac{e^{4t_2}}{|\varphi_\varepsilon(u)|^2}
         + 
         \frac{\espe[X_1^2]}{J t_1} + \frac{ \espe[\noise^2]}{J t_1^2} + 4 \frac{m}{(J t_1)^\eta} \\
        &\leq |\varphi_X(u)|^2 +  \frac{2 |\varphi_X(u)|^2}{1 + u^{2}}  \frac{1}{|\varphi_\varepsilon(u)|^2}
         + 
         \frac{\espe[X_1^2]}{J t_1} + \frac{ \espe[\noise^2]}{J t_1^2} + 4 \frac{\mmax}{(J t_1)^\eta} \\
         &
         \leq |\varphi_X(u)|^2 +  \frac{2 |\varphi_X(u)|^2}{(1 + u^{2})} \frac{(1+u^2)}{d}
         + 
         \frac{\espe[X_1^2]}{J t_1} + \frac{ \espe[\noise^2]}{J t_1^2} + 4 \frac{(t_2 - t_1)^{2\alpha}}{J \cdot t_1^\eta} \\
         &\leq C_0 |\varphi_X(u)|^2,
    \end{align*}
as $J \to \infty$. 
From now on, all that remains to be done is to look at the event $$\Acal = \{ \Big|\frac{ \varphi_X(u)}{\sqrt{1 + u^2}}  \Big| \leq \frac{2e^{2 t_2}}{\sqrt{J}(t_2 - t_1)} \}\eqfinp$$
By definition, we know that
    $$ \bar \varphi_X^J(u) \leq 1 + \frac{2 \log(J)}{(t_2 - t_1)}\eqfinp$$
Then, on the event $\Acal$
we see that 
        \begin{align*}
         \EE \big[ & \int_{[m, \hat m_n]} |\bar \varphi_X^J(u)  - \varphi(u)|^2 \cdot \1{\Acal} \big] 
         \\
         & 
         \leq
         \hspace{-0.8cm}
         \int\limits_{|u| \in [m, (\mmax)^\alpha]} \hspace{-0.8cm}
         |\varphi(u)|^2 \du + \Big(2 + \frac{2 \log(J)}{(t_2 - t_1)}\Big)^2 
         \int_{|u| \in [m, (\mmax)^\alpha]} 
         \PP \bigg( \Big|\frac{ \hat \varphi_X(u)}{\sqrt{1 + u^2}}  \Big| \geq \frac{\kappa_{J, t_1, t_2}}{\sqrt{J}(t_2 - t_1)} \bigg)\1{\Acal} \du \\
         & 
         \leq
         \hspace{-0.8cm}
         \int\limits_{|u| \in [m, (\mmax)^\alpha]}
         \hspace{-0.8cm}|\varphi(u)|^2 \du + \Big(2 + \frac{2 \log(J)}{(t_2 - t_1)}\Big)^2 
         \int_{|u| \in [m, (\mmax)^\alpha]} 
         \PP \bigg( \Big|\frac{ \hat \varphi_X(u)}{\sqrt{1 + u^2}}  \Big| \geq \frac{\kappa_{J, t_1, t_2}}{\sqrt{J}(t_2 - t_1)} \bigg) \1{\Acal} \du\\
         &
         \leq
         \hspace{-0.8cm}
         \int\limits_{|u| \in [m, (\mmax)^\alpha]} 
         \hspace{-0.8cm}|\varphi(u)^2| \du + \Big(2 + \frac{2 \log(J)}{(t_2 - t_1)}\Big)^2 
         \hspace{-0.8cm}\int\limits_{|u| \in [m, (\mmax)^\alpha]} \hspace{-0.8cm}
         \PP \bigg( \Big|\frac{ \hat \varphi_X(u) - \varphi(u)}{\sqrt{1 + u^2}}  \Big| \geq \frac{\kappa \sqrt{\ln(J(t_2 - t_1)^2)}}{\sqrt{J}(t_2 - t_1)}\bigg) \du \\ 
         &
         \leq
         \hspace{-0.8cm}
         \int\limits_{|u| \in [m, (\mmax)^\alpha]}
         \hspace{-0.8cm}|\varphi(u)^2| \du + \Big(2 + \frac{2 \log(J)}{(t_2 - t_1)}\Big)^2 \hspace{-0.8cm}
         \int\limits_{|u| \in [m, (\mmax)^\alpha]} \hspace{-0.8cm}
         \PP \bigg( \big| \hat \varphi_X(u) - \varphi(u)  \big| \geq \frac{\kappa \sqrt{\ln(J(t_2 - t_1)^2)}}{\sqrt{J}(t_2 - t_1)}\bigg) \du \\ 
         &
         \leq
         \hspace{-0.8cm}
         \int\limits_{|u| \in [m, (\mmax)^\alpha]} 
         \hspace{-0.8cm}
         |\varphi(u)^2| \du +  T_J \Big(2 + \frac{2 \log(J)}{(t_2 - t_1)}\Big)^2
         \eqfinp
    \end{align*}
    
Finally, it only remains to bound $T_J$ where 
    \begin{equation*}
         T_J = \hspace{-0.8cm} \int\limits_{|u| \in [m, (\mmax)^\alpha]} \hspace{-0.8cm}
         \PP \bigg( \big| \hat \varphi_X(u) - \varphi(u)  \big| \geq \frac{\kappa \sqrt{\ln(J(t_2 - t_1)^2)}}{\sqrt{J}(t_2 - t_1)}\bigg) \du
         \eqfinp
    \end{equation*}
By definition, we have

{\small
    \begin{align*}
    \PP \bigg( &\big| \hat \varphi_X(u) - \varphi(u)  \big|  \geq \frac{\kappa \sqrt{\ln(J(t_2 - t_1)^2)}}{\sqrt{J}(t_2 - t_1)}\bigg) \\
    &= 
    \PP \bigg( \frac{\big|  \log \hat \varphi_{Z_{t_2}}^{\, J}  (u) 
         - \log \hat \varphi_{Z_{t_1}}^{\, J}  (u) 
         - \log\varphi_{Z_{t_2}} (u) + \log \varphi_{Z_{t_1}} (u) \big|}{t_2 - t_1} \geq \frac{\kappa \sqrt{\ln(J(t_2 - t_1)^2)}}{\sqrt{J}(t_2 - t_1)}\bigg) \\
    &\leq
    \PP \bigg( \big| \log \hat \varphi_{Z_{t_1}}^{\, J}  (u) 
        - \log \varphi_{Z_{t_1}} (u) \big| +  \big|  \log \hat \varphi_{Z_{t_2}}^{\, J}  (u)
         - \log\varphi_{Z_{t_2}} (u) \big| \geq (t_2 - t_1) \frac{\kappa \sqrt{\ln(J(t_2 - t_1)^2)}}{\sqrt{J}(t_2 - t_1)}\bigg) \\
    &\leq 
    \PP \bigg( \big| \log \hat \varphi_{Z_{t_1}}^{\, J}  (u) 
        - \log \varphi_{Z_{t_1}} (u) \big| \geq (t_2 - t_1)  \frac{1}{2} \cdot \frac{\kappa \sqrt{\ln(J(t_2 - t_1)^2)}}{\sqrt{J}(t_2 - t_1)}\bigg) \\
        &\qquad \qquad \qquad + 
    \PP \bigg( \big| \log \hat \varphi_{Z_{t_2}}^{\, J}  (u) 
        - \log \varphi_{Z_{t_2}} (u) \big| \geq  (t_2 - t_1)  \frac{1}{2} \cdot \frac{\kappa \sqrt{\ln(J(t_2 - t_1)^2)}}{\sqrt{J}(t_2 - t_1)}\bigg) 
        \eqfinp
    \end{align*}  
}

Without loss of generality, we only consider the term
    \begin{equation*}
        \PP \bigg( \big| \log \hat \varphi_{Z_{t_2}}^{\, J}  (u) 
        - \log \varphi_{Z_{t_2}} (u) \big| \geq (t_2 - t_1)  \frac{\kappa \sqrt{\ln(J(t_2 - t_1)^2)}}{2 \sqrt{J}(t_2 - t_1)}\bigg)
        \eqfinp
    \end{equation*}
    
By taking $\gamma_\Delta$ and $\zeta > 0$, like in the proof of Theorem \ref{thm:analysis_error}, we have that 
    \begin{align*}
        \PP \bigg( \big| &\log \hat \varphi_{Z_{t_2}}^{\, J}  (u) 
        - \log \varphi_{Z_{t_2}} (u) \big| \geq (t_2 - t_1)  \frac{\kappa \sqrt{\ln(J(t_2 - t_1)^2)}}{2 \sqrt{J}(t_2 - t_1)}\bigg) \\
        &\leq
        \PP \bigg( \big| \hat \varphi_{Z_{t_2}}^{\, J}  (u) 
        -  \varphi_{Z_{t_2}} (u) \big| \geq (t_2 - t_1)  | \varphi_{Z_{t_2}} (u)| \frac{\kappa \sqrt{\ln(J(t_2 - t_1)^2)}}{2 \sqrt{J}(t_2 - t_1)}\bigg) + \PP(\Omega_{\zeta, (t_2 - t_1)^2}^c)
        \eqfinp
    \end{align*}

Using hypothesis on the regularity of $|\varphi_{\varepsilon}|$, we have that
    \begin{equation*}
        |\varphi_{Z_{t_2}}| 
        = |\varphi_{Y_{t_2}}| \cdot |\varphi_{\varepsilon}| 
        \geq |\varphi_{Y_{t_2}}| \cdot \frac{d}{\sqrt{1 + u^2}} 
        \geq e^{2 t_2} \cdot \frac{d}{\sqrt{1 + (\mmax)^2}} 
        \eqfinp
    \end{equation*}

Let $c(\theta) = \kappa (t_2 - t_1)  e^{2 t_2} \cdot \frac{d}{\sqrt{1 + (\mmax)^2}}$.

It follows from Hoeffding inequality and Lemma \ref{lm:main1} that 
    \begin{align*}
        \PP \bigg( \big| \log \hat \varphi_{Z_{t_2}}^{\, J}  (u) 
        &- \log \varphi_{Z_{t_2}} (u) \big| \geq \frac{\kappa \sqrt{\ln(J(t_2 - t_1)^2)}}{2 \sqrt{J}(t_2 - t_1)}\bigg) \\
        &\leq
        2(J(t_2 - t_1)^2)^{-c(\theta)^2} + \frac{\espe[X_1^2]}{J \tdiff^2} + \frac{ \espe[\noise^2]}{J \tdiff^4} + 4 \frac{\mmax}{(J \tdiff^2)^\eta}
        \eqfinp
    \end{align*}

If we take $\eta > 3$ such that $2\alpha - \eta < -1$ and $\zeta > \sqrt{\tdiff^2(1 + 2 \eta) }$, then we have
    \begin{equation*}
    \begin{split}
             \int\limits_{|u| \in [m, (\mmax)^\alpha]}\PP \bigg( \big| \log \hat \varphi_{Z_{t_2}}^{\, J}  (u) 
        - \log \varphi_{Z_{t_2}} (u) \big|
        \geq \frac{\kappa \sqrt{\ln(J(t_2 - t_1)^2)}}{2 \sqrt{J}(t_2 - t_1)}\bigg) 
       \\
     \leq 4 (J \tdiff^2)^{\alpha - c(\theta)^2} + \frac{C'}{J \tdiff^2} + \frac{C'}{J \tdiff^4}  
    \end{split}
    \end{equation*}
    
where $C'$ depends on $\EE[X_1^2], \EE[\varepsilon^2]$. A similar reasoning on 
$$              \int\limits_{|u| \in [m, (\mmax)^\alpha]}\PP \bigg( \big| \log \hat \varphi_{Z_{t_1}}^{\, J}  (u) 
        - \log \varphi_{Z_{t_1}} (u) \big|
        \geq \frac{\kappa \sqrt{\ln(J(t_2 - t_1)^2)}}{2 \sqrt{J}(t_2 - t_1)}\bigg), $$
        leads to 
    $$ T_J \leq C_0 (J \tdiff^2)^{\alpha - c(\theta)^2} + \frac{C_1}{J \tdiff^2} + \frac{C_2}{J \tdiff^4}  $$
where $C_0, C_1$ and $C_2$ depends on $\EE[X_1^2], \EE[\varepsilon^2]$.

\subsection{Proof of Theorem \ref{T:risk}}
\label{S:ProofMain} This proof is an adaptation of the proof of Theorem 3.1 of Duval and Kappus \cite{duval2019adaptive} for the settings of the multiplicative decompounding.\\

 We have the decomposition
    \begin{equation*}
    \|\hat{f_m}  - f \|^2 _{\omega_1}
    \leq \|f_m  - f \|^2_{\omega_1} + \|\hat{f_m}  - f_m \|^2_{\omega_1}
    \eqfinp
    \end{equation*}
To prove the theorem, we have to bound the variance term $\|\hat{f_m}  - f_m \|^2_{\omega_1}$. \\
Using the isometry equation, we get 
    \begin{equation*}
    \|\hat{f_m} - f_m \|^2_{\omega_1}
    = 
   \frac{1}{2\pi} \int_{-m}^m \big|  \tilde{\Mcal_1[f]}(s) -  \Mcal_1[f](s) \big|^2 \diff s
   \eqfinp
    \end{equation*}

To get a majorization of the right side term, we consider different events on which it can be control. \\

Fix $\gamma > \zeta$ and set $\gamma_\Delta = \frac{2\zeta}{1 \wedge \Delta}$. We consider the events
\begin{equation*}
    \Omega_{\zeta, \Delta}(m) = \Bigg\{ 
            \forall u \in [-m, m], |\hat{\Mcal_1[\Delta]}(u) - \Mcal_1[\Delta](u) |
            \leq
            \zeta \sqrt{\frac{\log(n \Delta)}{n \Delta}}
        \Bigg\}
\end{equation*}
and
\begin{equation*}
    \{m \leq M_{n, \Delta}^{(\gamma_\Delta)}\} 
\mtext{where}
M_{n, \Delta}^{(\gamma_\Delta)} = \min \Big\{ u \geq 0: |\Mcal_1[\Delta](u) | = \gamma c_m \sqrt{\log(Jt)/(Jt)} \Big\}
        \eqfinp
    \end{equation*}

We have
\begin{equation*}
\begin{split}
     \frac{1}{2\pi} \int_{-m}^m \big|  \tilde{\Mcal_1[f]}(s) &-  \Mcal_1[f](s) \big|^2 \diff s
     = 
    \frac{1}{2\pi} \int_{-m \wedge M_{n, \Delta}^{(\gamma_\Delta)}}^{m \wedge M_{n, \Delta}^{(\gamma_\Delta)}} \ind{\Omega_{\zeta, \Delta}(m)} \big|  \tilde{\Mcal_1[f]}(s) -  \Mcal_1[f](s) \big|^2 \diff s
      \\
      &+
    \frac{1}{2\pi} \ind{m > M_{n, \Delta}^{(\gamma_\Delta)}} \cdot \ind{\Omega_{\zeta, \Delta}(m)} \int_{|s| \in [M_{n, \Delta}^{(\gamma_\Delta)}, m ]} \big|  \tilde{\Mcal_1[f]}(s) -  \Mcal_1[f](s) \big|^2 \diff s
    \\ &+
    \ind{\Omega_{\zeta, \Delta}^c(m)} \frac{1}{2\pi} \int_{-m}^m \big|  \tilde{\Mcal_1[f]}(s) -  \Mcal_1[f](s) \big|^2 \diff s
\end{split}
\end{equation*}

Focus on the event $\Acal = \{ |u| \leq m \wedge M_{n, \Delta}^{(\gamma_\Delta)} \}  \cap \Omega_{\zeta, \Delta}(m)$. First, we prove that $$\hat{\Mcal_1[f]}(s) = \tilde{\Mcal_1[f]}(s).$$

The triangle inequality ensures that
\begin{equation*}
    \hat{\Mcal_1[f]}(s) \leq 1 + \frac{|\hat{\log \Mcal_1[\Delta]}(s) - \log \Mcal_1[\Delta](s)| + |\log \Mcal_1[\Delta](s)|}{\Delta}
    \eqfinp
\end{equation*}

and Lemma~\ref{lm:main2} ensures that 
\begin{equation*}
    \hat{\Mcal_1[f]}(s) \leq 1 + \frac{1}{\Delta} \log \Big( \frac{\gamma}{\gamma - \zeta}\Big)
    \frac{|\hat{\Mcal_1[\Delta]}(s) - \Mcal_1[\Delta](s)|}{|\Mcal_1[\Delta]|}
        + \frac{|\log \Mcal_1[\Delta]|}{\Delta}
    \eqfinp
\end{equation*}
Furthermore, since 
\begin{equation*}
    \Mcal_1[f](s) = 1 + \frac{1}{\Delta} \log \Mcal_1[\Delta](s)
    \eqfinv
\end{equation*}
it follows that
\begin{equation*}
    \hat{\Mcal_1[f]}(s) \leq 3 + \frac{1}{\Delta} \log \Big( \frac{\gamma}{\gamma - \zeta}\Big) \leq 4
    \eqfinp
\end{equation*}
Then $\hat{\Mcal_1[f]}(s) = \tilde{\Mcal_1[f]}(s)$. \\

Therefore,
\begin{align*}
    \EE \Big[\frac{1}{2\pi} \int_{-m \wedge M_{n, \Delta}^{(\gamma_\Delta)}}^{m \wedge M_{n, \Delta}^{(\gamma_\Delta)}} &\ind{\Omega_{\zeta, \Delta}(m)} \big|  \tilde{\Mcal_1[f]}(s) 
     -  \Mcal_1[f](s) \big|^2 \diff s \Big]   \\
    &
    = \frac{1}{2\pi} \int_{-m \wedge M_{n, \Delta}^{(\gamma_\Delta)}}^{m \wedge M_{n, \Delta}^{(\gamma_\Delta)}} \EE \Big[ \ind{\Omega_{\zeta, \Delta}(m)} \big|  \hat{\Mcal_1[f]}(s) -  \Mcal_1[f](s) \big|^2 \Big]\diff s \\
    &
    = \frac{1}{2\pi \Delta^2} \int_{-m \wedge M_{n, \Delta}^{(\gamma_\Delta)}}^{m \wedge M_{n, \Delta}^{(\gamma_\Delta)}} \EE \Big[ \ind{\Omega_{\zeta, \Delta}(m)} \big|  \hat{ \log \Mcal_1[\Delta]}(s) -  \log \Mcal_1[\Delta](s) \big|^2 \Big]\diff s 
    \eqfinp
\end{align*}
Lemma \ref{lm:main2} ensures that 
\begin{equation*}
    \begin{split}
       \EE \Big[\frac{1}{2\pi} \int_{-m \wedge M_{n, \Delta}^{(\gamma_\Delta)}}^{m \wedge M_{n, \Delta}^{(\gamma_\Delta)}}  \ind{\Omega_{\zeta, \Delta}(m)} \big|  \tilde{\Mcal_1[f]}(s) 
     -  \Mcal_1[f](s) \big|^2 \diff s \Big]  \\  
    \leq
    \frac{1}{2\pi \Delta^2} \int_{-m \wedge M_{n, \Delta}^{(\gamma_\Delta)}}^{m \wedge M_{n, \Delta}^{(\gamma_\Delta)}}\frac{  \EE \big[|\hat{\Mcal_1[\Delta]}(s) - \Mcal_1[\Delta](s)|^2 \big]}{|\Mcal_1[\Delta](s)|^2}\diff s
    \eqfinp
     \end{split}
\end{equation*}
We have
\begin{equation*}
     \EE \big[|\hat{\Mcal_1[\Delta]}(s) - \Mcal_1[\Delta](s)|^2 \big]
     = \var(\hat{\Mcal_1[\Delta]}(s)) 
     \leq \frac{1}{n} \EE(|Z_\Delta    ^{c-1 + it}|^2)
     \eqfinp
\end{equation*}
Then we have
\begin{equation*}
    \EE \Big[\frac{1}{2\pi} \int_{-m \wedge M_{n, \Delta}^{(\gamma_\Delta)}}^{m \wedge M_{n, \Delta}^{(\gamma_\Delta)}} \ind{\Omega_{\zeta, \Delta}(m)} \big|  \tilde{\Mcal_1[f]}(s) 
    -  \Mcal_1[f](s) \big|^2 \diff s \Big]    
    \leq 
    \frac{1}{2\pi n \Delta^2} \int_{-m}^{m}\frac{1}{|\Mcal_1[\Delta](s)|^2}\diff s 
    \eqfinp
\end{equation*}

Then, set $\zeta > \sqrt{5 \Delta}$. Using Lemma~\ref{lm:main1} with $\eta = 2$ to get
\begin{equation*}
    \EE \Big[\ind{\Omega_{\zeta, \Delta}^c(m)} \frac{1}{2\pi} \int_{-m}^m \big|  \tilde{\Mcal_1[f]}(s) -  \Mcal_1[f](s) \big|^2 \diff s \Big]
    \leq \frac{50 m}{2\pi} \Big(\frac{\EE[\ln(X_1)^2]}{\Delta n} + 4 \frac{m}{(n \Delta)^{2}} \Big)
    \eqfinp
\end{equation*}

Finally, we observe that $|\Mcal_1(\Delta)(s)|> e^{-2\Delta}$. Following the strategy describes in Duval Kappus \cite{duval2019adaptive}, we obtain that $M_{n, \Delta}^{\gamma_\Delta} = \infty$ and that 
\begin{equation*}
    \EE \Big[\frac{1}{2\pi} \ind{m > M_{n, \Delta}^{(\gamma_\Delta)}} \cdot \ind{\Omega_{\zeta, \Delta}(m)} \hspace{-0.5cm} \int_{|s| \in [M_{n, \Delta}^{(\gamma_\Delta)}, m ]} \limits\hspace{-0.5cm} \big|  \tilde{\Mcal_1[f]}(s) -  \Mcal_1[f](s) \big|^2 \diff s \Big]
    = 0
    \eqfinp
\end{equation*}

\subsection{Proof of Theorem~\ref{T:adapt}}

Let $0 < m < (n\Delta)^\alpha$. As for the \cref{thm:oracle}, the proof is divided in two steps: in the first one, we control $\EE \big[ \|\bar{f_{\hat{m_n}, \Delta}}  - f \|_{\omega_1}^2\big]$ on the event $\Ecal = \{ \hat m_n < m\}$, then we control it on the complementary event $\Ecal^c = \{ \hat m_n \geq m\}$.\\ 

\paragraph{\textbf{Step 1} -- The first event $\Ecal$}
The triangle inequality and the isometry equality leads to 
{\small
\begin{align*}
    \EE \Big[ \1{\Ecal} \hspace{-0.4cm} \int_{|u| \in [\hat{m}_n, m]} \limits\hspace{-0.5cm} |\Mcal_1[f](u)|^2 \diff u \Big]
    &\leq 2 \EE \Big[ \1{\Ecal} \hspace{-0.4cm} \int_{|u| \in [\hat{m}_n, m]} \limits\hspace{-0.5cm} |\tilde\Mcal_1[f](u)|^2 \diff u \Big]  + 2 \EE \Big[ \1{\Ecal} \hspace{-0.4cm} \int_{|u| \in [\hat{m}_n, m]} \limits\hspace{-0.5cm} |\tilde\Mcal_1[f](u) - \Mcal_1[f](u)|^2 \diff u \Big]  \\
    &\leq 2 \EE \Big[ \1{\Ecal} \hspace{-0.4cm} \int_{|u| \in [\hat{m}_n, m]} \limits\hspace{-0.5cm} \frac{  \kappa_{n,\Delta}^2}{n\Delta} \diff u \Big]  + 2 \EE \Big[ \1{\Ecal} \hspace{-0.4cm} \int_{|u| \in [0, m]} \limits\hspace{-0.5cm} |\tilde\Mcal_1[f](u) - \Mcal_1[f](u)|^2 \diff u \Big]  \\
    &\leq 4 \frac{\kappa_{n,\Delta}^2 m}{n\Delta} + \frac{1}{\pi n \Delta^2} \int_{-m}^{m}\frac{1}{|\Mcal_1[\Delta](s)|^2}\diff s 
    + \frac{50}{\pi}m \Big(\frac{\EE[\ln(X_1)^2]}{\Delta n} + 4 \frac{m}{(n \Delta)^{2}} \Big)
    \eqfinp
\end{align*}
}
By definition, $\kappa_{n,\Delta} = e^{2\Delta} + \kappa \sqrt{\log(n\Delta)}$. It follows from \cref{T:risk} that
\begin{multline*}
    \EE \big[ \1{\Ecal}\|\bar{f_{\hat{m_n}, \Delta}}  - f \|_{\omega_1}^2\big] \leq \| f_m - f \|_{\omega_1}^2 + C\frac{\log(n\Delta)m}{n\Delta} \\
    + \frac{1}{\pi n \Delta^2} \int_{-m}^{m}\frac{1}{|\Mcal_1[\Delta](s)|^2}\diff s 
    + \frac{50}{\pi}m \Big(\frac{\EE[\ln(X_1)^2]}{\Delta n} + 4 \frac{m}{(n \Delta)^{2}} \Big)
    \eqfinp
\end{multline*}

\vspace{1em}
\paragraph{\textbf{Step 2} -- The second event $\Ecal^c$} Henceforth, we considere the case where $\hat{m}_n \geq m$. It remains to control the surplus in the variance of $\tilde f_{\hat m_n}$.

Since $m \leq \hat{m}_n \leq (n\Delta)^\alpha$, it follows that
\begin{align*}
    \EE \Big[ \1{\Ecal^c} \hspace{-0.5cm} \int_{|u| \in [m, \hat m_n]} \limits \hspace{-0.5cm} \big|\bar \Mcal_1[f](u) - \Mcal[f](u)\big|^2 \Big]
    \leq
    \Big[ \hspace{-0.5cm} \int_{|u| \in [m, (n\Delta)^\alpha]}\limits \hspace{-0.5cm} \EE  \big[\big|\bar \Mcal_1[f](u) - \Mcal[f](u)\big|^2 \big]
    \eqfinp
\end{align*}

Let $\eta > 2$ such that $\alpha-\eta < -1$. \cref{lm:Mellin1} and \cref{lm:Mellin2} ensure that
\begin{align*}
    \EE  \big[\big|\bar \Mcal_1[f](u) - \Mcal[f](u)\big|^2 \big]
    &\leq |\Mcal_1[f](s)|^2 + \EE  \big[\big|\hat \Mcal_1[f](u) - \Mcal[f](u)\big|^2 \big]  \\
    &\leq  |\Mcal_1[f](s)|^2 + \frac{2}{n \Delta^2 \big|\Mcal_1[\Delta](s)\big|^2}
    + \frac{\EE[\ln(X_1)^2]}{\Delta n} + \frac{4}{n \Delta}
    \eqfinp
\end{align*}

First, we assume that $\{|\Mcal_1[f](s)|^2 > \frac{e^{2\Delta}}{\sqrt{n\Delta}}\}$. Given that $\Mcal_1[\Delta](u) > e^{-2\Delta}$, it follows in this case that 
\begin{align*}
    \EE  \big[\big|\bar \Mcal_1[f](u) - \Mcal_1[f](u)\big|^2 \big]
    \leq |\Mcal_1[f](u)|^2 (6 + \EE[\ln(X_1)^2]). 
\end{align*}

Therefore, 
\begin{align*}
    \EE \Big[ \hspace{-0.5cm} \int_{|u| \in [m, \hat m_n]} \limits \hspace{-0.5cm} \big|\bar \Mcal_1[f](u) - \Mcal[f](u)\big|^2 \cdot \1{\{||\Mcal_1[f](s)|^2| > \frac{e^{2\Delta}}{\sqrt{n\Delta}}\}} \cdot \1{\Ecal^c} \Big] \diff u 
    &\leq A \hspace{-0.5cm} \int_{|u| \in [m, \hat m_n]} \limits \hspace{-0.5cm} |\Mcal_1[f](u)|^2 \diff u \\
    &\leq A \hspace{-0.3cm} \int_{[-m, m]^c} \limits \hspace{-0.3cm} |\Mcal_1[f](u)|^2 \diff u
    \eqfinv
\end{align*}
where $A = 6 + \EE[\ln(X_1)^2]$.\\

Secondly, we assume that $\{|\Mcal_1[f](s)|^2 \leq \frac{e^{2\Delta}}{\sqrt{n\Delta}}\}$. By construction, $|\bar \Mcal_1[f](s)| \leq 4$. It follows that 
\begin{align*}
    \EE \Big[ \hspace{-0.5cm} \int_{|u| \in [m, \hat m_n]} \limits \hspace{-0.5cm} \big|
    & \bar \Mcal_1[f](u)- \Mcal[f](u)\big|^2 
    \cdot \1{\{|\Mcal_1[f](s)|^2| \leq \frac{e^{2\Delta}}{\sqrt{n\Delta}}\}} 
    \cdot \1{\Ecal^c} \Big] \diff u \\
    &\leq \hspace{-0.5cm} \int_{|u|\in[m, (n\Delta)^\alpha]} \limits\hspace{-0.5cm} |\Mcal_1[f](u)|^2 \diff u 
        + 25 \hspace{-0.6cm} \int_{|u| \in [m, (n\Delta)^\alpha]} \limits\hspace{-0.5cm} \PP\Big( |\hat \Mcal_1[f]| \geq \frac{\kappa_{n, \Delta}}{\sqrt{n\Delta}} \Big) \cdot \1{\{|\Mcal_1[f](s)|^2| \leq \frac{e^{2\Delta}}{\sqrt{n\Delta}}\}} \diff u \\
    &\leq \hspace{-0.5cm} \int_{|u|\in[m, (n\Delta)^\alpha]} \limits\hspace{-0.5cm} |\Mcal_1[f](u)|^2 \diff u 
    + 25 \hspace{-0.6cm} \int_{|u| \in [m, (n\Delta)^\alpha]} \limits\hspace{-0.5cm} \PP\Big( |\hat \Mcal_1[f] - \Mcal_1[f]| \geq \kappa \sqrt{\log(n\Delta/(n\Delta))} \Big)  \diff u
    \eqfinp
\end{align*}

It remains to bound the last term. For that, we use \cref{lm:Mellin1} and \cref{lm:Mellin2} by taking $\gamma_\Delta$ and $\zeta > 0$, like in the proof of \cref{T:risk}. By definition, we have that 
\begin{align*}
    \PP\Big( |\hat \Mcal_1[f] &- \Mcal_1[f]| \geq \kappa \sqrt{\log(n\Delta/(n\Delta))} \Big)  \\
    &= 
    \PP\Big( |\log \hat \Mcal_1[\Delta] - \log \Mcal_1[\Delta]| \geq \kappa \Delta \sqrt{\log(n\Delta/(n\Delta))} \Big)  \\
    &\leq 
    \PP\Big( | \hat \Mcal_1[\Delta] - \Mcal_1[\Delta]| \geq |\Mcal_1[\Delta]|\kappa \Delta \sqrt{\log(n\Delta/(n\Delta))} \Big)  + \PP(\Omega_{\zeta, \Delta}^c((n\Delta)^\alpha)).
\end{align*}

Let $c(\Delta) = k\Delta e^{-2\Delta}$. The Hoeffding inequality and \cref{lm:Mellin1} ensure that 
\begin{align*}
    \PP\Big( |\hat \Mcal_1[f] &- \Mcal_1[f]| \geq \kappa \sqrt{\log(n\Delta/(n\Delta))} \Big) 
    \leq
    2(n\Delta)^{\alpha - c(\Delta)^2} + \frac{\EE[\ln(X_1)^2]}{n\Delta} + 4(n\Delta)^{\alpha-\eta}.
\end{align*}
If we take $\eta> 3$ such that $2\alpha - \eta < -1$ and $\zeta > \sqrt{\Delta(1+2\eta)}$, then 
\begin{align*}
   \int_{|u| \in [m, (n\Delta)^\alpha]} \limits\hspace{-0.5cm} \PP\Big( |\hat \Mcal_1[f] - \Mcal_1[f]| \geq \kappa \sqrt{\log(n\Delta/(n\Delta))} \Big)  \diff u
    \leq 4(n\Delta)^{\alpha - c(\Delta)^2} + \frac{B}{n\Delta}
\end{align*}
where $B$ depends on $\EE[\ln(X_1)^2]$.\\

It follows that there exist two constant $A,B$ such that
{\begin{multline*}
    \EE \big[ \|\bar{f_{\hat{m_n}, \Delta}}  - f \|_{\omega_1}^2\big] \\
    \leq A \Big(\| f_m - f \|_{\omega_1}^2 + \frac{\log(n\Delta)m}{n\Delta}
    + \frac{1}{n \Delta^2} \int_{-m}^{m} \limits\frac{1}{|\Mcal_1[\Delta](s)|^2}\diff s 
    + 4 \frac{m^2}{(n \Delta)^{2}} \Big)
    \\
    + B\Big((n\Delta)^{\alpha - c(\Delta)^2} + \frac{1}{n\Delta}\Big)
    \eqfinp
\end{multline*}}

We conclude by taking the infimum in $m$.

\section{Conclusion and perspectives}

In this article, we develop an adaptive procedure to estimate the jump density $f$ in a noisy compound Poisson process. In our case, we observe several noisy compound Poisson processes. We have shown that by looking at what happens at two different times, it is possible to reconstruct the density $f$. \\

In practice, the experimental data allow us to observe the process at different times. A discussion could be conducted later to understand if it is possible to combine estimators built with different $t_1$ and $t_2$ to improve our estimator.

In a future work, we plan to apply our statistical method to biological data from the article by Robert et al. \cite{robert2018mutation}. \\


\appendix

\section{The distinguished logarithm}
\label{S:DistingLog}

In the section, we recall the definition and some properties of the distinguished logarithm. The reader who would like to have more information can refer to the articles of Duval and Kappus \cite{duval2017nonparametric, duval2019adaptive} or to the article of Finkelstein et al. \cite{finkelstein1999extinguishing}.

\begin{theorem}
\label{duval2017nonparametric_lm1}
Let $d \in \NN - \{0\}$ and let $\varphi \in \Ccal^0(\RR^d; \CC^*)$ be a continuous application which never vanishes and such that $\varphi(0) = 1$. \\
Then there exists a unique continuous application $\psi\in \Ccal^0(\RR^d; \CC)$ such that
	\begin{enumerate}
	\item $\psi(0) = 0$,  
	\item for all $x \in \RR^d \eqsepv \varphi(x) = e^{\psi(x)}.$
	\end{enumerate}
The application $\psi$ is called the Distinguished Logarithm of $\varphi$ and is denoted by $\log \varphi$.
\end{theorem}

\begin{remark} In general, the distinguished logarithm does not reduce to a composition by the principal determination of the logarithm. In his book \cite{sato1999levy}, Sato remarks to his readers that one can have
    $\varphi(z_1) = \varphi(z_2) \mtext{and} \log \varphi(z_1) \neq \log \varphi(z_2)
    \eqfinp $ \\
Indeed, consider the application $\varphi(t) = e^{it}, (t \in \R)$. It verifies all the assumptions of Theorem \ref{duval2017nonparametric_lm1}.
which ensures that
    \begin{equation*}
        \log \varphi(t) = it 
        \eqsepv (t \in \R)
        \eqfinv
    \end{equation*}
It follows that
    \begin{equation*}
        \varphi(0) = \varphi(2) = 1 \mtext{and} \log\varphi(0) = 0 \eqsepv \log\varphi(2) = 2i \pi
        \eqfinp
    \end{equation*}
\end{remark}

\begin{proposition}(Theorem 2. of \cite{finkelstein1999extinguishing}) 
\label{finkelstein1999extinguishing_thm2}\\
Let $d > 0$ be a positive integer and $f, (f_n)_{n \in \NN} \in \Ccal^0(\RR^d; \CC^*)$ be continuous functions which never vanishes and such that $f(0) = f_n(0)= 1$ for all $n \in \NN$. If $(f_n)$ converges uniformly to $f$ on compact subsets of $\RR^d$, then the sequence $(\log(f_n))_{n \in \NN}$ converges uniformly to $\log(f)$ on compact subsets of $\RR^d$. 
\end{proposition}

\begin{proposition}(Lemma 3. of \cite{duval2017nonparametric})
\label{duval2017nonparametric_lm3}\\
Let $\varphi$ be a characteristic function without zeroes and assume that $\varphi$ is differentiable. Then, it follows that
	\begin{equation*}
	\log \varphi (u ) = \int_0^u \frac{\varphi'(z)}{\varphi(z)} \dz
	\eqfinp
	\end{equation*}
\end{proposition}

\begin{corollary}
Let $d \in \NN - \{0\}$ and $f_1, f_2 \in \Ccal^0(\RR^d; \CC^*)$ be two continuous functions which never vanishes and such that $f_1(0) = f_2(0) = 1$. Assume that $f_1, f_2$ are both differentiable and let $h: \RR^d \to \CC^* $ denotes the quotient $h = \frac{f_2}{f_1}$. 
Then, 
	\begin{equation*}
	\forall u \in \RR^d
	\eqsepv
	\log h(u) = \log f_2(u) - \log f_1(u)
	\eqfinp
	\end{equation*}
\end{corollary}
\begin{proof}
\begin{align*}
    \log h(u) 
    &= \int_0^u \frac{h'(z)}{h(z)} \dz
    = \int_0^u \frac{\big(\frac{f_2}{f_1})'(z)}{\big(\frac{f_2}{f_1})(z)} \dz
    =
    \int_0^u \frac{\big(\frac{f_2' f_1 - f_2 f_1'}{f_1^2})}{\big(\frac{f_2}{f_1})} \dz
    =    \int_0^u \Big(\frac{f_2' f_1 - f_2 f_1'}{f_1 f_2}\Big) \dz  \\
    &=  \int_0^u \Big(\frac{f_1' }{f_1}\Big) \dz -  \int_0^u \Big(\frac{ f_2'}{f_2}\Big) \dz 
    = \log f_2(u) - \log f_1(u) 
    \eqfinp
\end{align*}
\end{proof}

\section{Useful lemmas}
\label{s:alem}

\subsection{Lemmas for decompounding with unknown noise}
\begin{lemma}\label{lm:main1}
Let $m, \zeta, t \in [0, \infty)$ be positive reals. We consider the event
    \begin{equation*}
        \Omega_{\zeta, t}(m) = \Bigg\{ 
            \forall u \in [-m, m], |\hat \varphi^J_{Z_{t}}(u) - \fctcar{Z_{t}}(u)|
            \leq
            \zeta \sqrt{\frac{\log(J t)}{J t}}
        \Bigg\}
        \eqfinp
    \end{equation*}
    
        If $\E[X_1^2] < \infty$, then
            \begin{equation*}
                \forall \eta > 0 
                \eqsepv
                \forall \zeta > \sqrt{t(1 + 2 \eta)}
                \eqsepv
                \proba(\Omega_{\zeta, t}(m)^c) \leq \frac{\espe[X_1^2]}{J t} + \frac{ \espe[\noise^2]}{J t^2} + 4 \frac{m}{(J t)^\eta}
                \eqfinp
            \end{equation*}
\end{lemma}
\begin{proof}
Let $c, h, \tau \in [0, \infty)$ be some positive reals. We define the events
    \begin{align*}
        A(c) 
        &= \left\{ \bigg| \frac{1}{J} \sum_{j = 1}^J |Z_t^j| - \E\bc{|Z_t|}  \bigg| \leq c \right\} 
        \eqfinv
        \\
        B_{h, \tau(m)} 
        &= \left\{ \forall |k| \leq \bigg\lceil \frac{m}{h} \bigg\rceil, |\hat \varphi^J_{Z_{t}}(k h) - \varphi_{Z_{t}}(k h) | \leq \tau \sqrt{\frac{\log (Jt)}{J t}}
        \right\}   \eqfinp \end{align*}
As we know that the function $x \mapsto e^{iux}$ is 1-Lipschitz, then for all $u \in \RR$ and for all $h \in [0, \infty)$ we have
    \begin{subequations}
    \begin{align*}
        | \hat\varphi^J_{Z_t}(u) - \hat \varphi^J_{Z_t}(u + h)|
        & = \Big| \frac{1}{J}\sum_{j = 1}^J e^{i u Z_t^j} - \frac{1}{J}\sum_{j = 1}^J e^{i (u+h) Z_t^j} \Big| \\
        &\leq  \frac{1}{J}\sum_{j = 1}^J \Big| e^{i u Z_t^j} - e^{i (u+h) Z_t^j} \Big| \\
        &\leq  \frac{h}{J}\sum_{j = 1}^J |Z_t^j| \\
        &= h \Big( \frac{1}{J} \sum_{j = 1}^J \Big[ |Z_t^j| - \E[|Z_t|] \Big] +  \E[|Z_t| \Big)
        \eqfinp
    \end{align*}
    \end{subequations}
The definition of $A(c)$ ensures that 
    \begin{equation}
        \forall u \in \RR
        \eqsepv
        h > 0
        \eqsepv
        | \hat\varphi^J_{Z_t}(u) - \hat \varphi^J_{Z_t}(u + h)| \1{A(c)}
        \leq
        h ( c + \E[|Z_t|)
        \eqfinp
    \label{eq:duval_5.1}
    \end{equation}
Moreover we have
    \begin{equation}
        \E[|Z_t|] \leq t \E[|X_1|] + \E[|\noise|]
        \eqfinp
    \end{equation}
    
If $\espe[X_1^2] < \infty$, by appling the Markov inequalities \\
    $$\var[|Z_t|] \leq \var[Z_t] = t \espe[X_1^2] + \espe[\noise^2]\eqfinv$$ 
we claim that
    \begin{equation}
        \proba(A(c)^c) \leq \frac{t \E[X_1^2]}{c^2 J} + \frac{\E[\noise^2]}{c^2 J}
        \eqfinp
        \label{eq:duval_eq5.2}
    \end{equation}


Moreover we have
    \begin{subequations}
    \begin{align}
        \proba(B_{h, \tau(m)}^c) 
        &=  \proba \left( \exists |k| \leq \bigg\lceil \frac{m}{h} \bigg\rceil, |\hat \varphi^J_{Z_{t}}(k h) - \varphi_{Z_{t}}(k h) | > \tau \sqrt{\frac{\log (Jt)}{J t}} \right) \\
        \onlyME{\intertext{It is well-known that the probability of a union of events is smaller than the sum of their probabilities, so we have}}
        &\leq \sum_{k = - \big\lceil \frac{m}{h} \big\rceil}^{\big\lceil \frac{m}{h} \big\rceil} \proba \bigg( |\hat \varphi^J_{Z_{t}}(k h) - \varphi_{Z_{t}}(k h) | > \tau \sqrt{\frac{\log (Jt)}{J t}} \bigg)  \\
        &\leq 
            \sum_{k = - \big\lceil \frac{m}{h} \big\rceil}^{\big\lceil \frac{m}{h} \big\rceil} \proba \Bigg( \Big| \sum_{j = 1}^J  e^{ikhZ_t^j} - \E[e^{ikhZ_t^j}] \Big| > \tau \sqrt{\frac{J\log (Jt)}{ t}} \Bigg) 
            \eqfinp
        \\
        \intertext{As $| e^{ikhZ_t^j}| \leq 1$ almost surely,Hoeffding's inequality ensures that }
        &\leq \sum_{k = - \big\lceil \frac{m}{h} \big\rceil}^{\big\lceil \frac{m}{h} \big\rceil} 2 \exp \Big( - \frac{\tau^2\log(Jt)}{2 t}\Big) \\
        &= 4 \Big\lceil \frac{m}{h} \Big\rceil (Jt)^{-\tau^2 / (2t)}
\eqfinp
    \end{align}
    \end{subequations}

Let $|u| \leq m$. Then there exists $k$ a positive integer such that $u \in [kh - \frac{h}{2}, kh + \frac{h}{2}]$. Then
    \begin{equation}
    \begin{split}
        \1{A(c) \cap B_{h, \tau}(m)} | \hat \varphi^J_{Z_{t}}(u) - \fctcar{Z_{t}}(u)| 
        &\leq \1{A(c) \cap B_{h, \tau}(m)} \Big( | \hat \varphi^J_{Z_{t}}(u) -  \hat \varphi^J_{Z_{t}}(kh) |  \\
        &\qquad +  | \hat \varphi^J_{Z_{t}}(kh) -  \varphi_{Z_{t}}(kh) | + | \varphi_{Z_{t}}(kh) -  \varphi_{Z_{t}}(u) | \Big)
\eqfinp
    \end{split}
    \label{eq:hq}
    \end{equation}
    
Since the function $x \mapsto e^{ixt}$ is 1-Lipschitz, we have
    \begin{subequations}
    \begin{align*}
        | \varphi_{Z_{t}}(kh) -  \varphi_{Z_{t}}(u) |
         &= |\E[e^{i kh Z_t} - e^{i u Z_t} ] \\
         &\leq h\E[|Z_t|]
         \eqfinp
    \end{align*}
    \end{subequations}

By applying this last result to Equation~\eqref{eq:hq}
\begin{equation*}
    \begin{split}
        \1{A(c) \cap B_{h, \tau}(m)} \sup_{|u| \leq m}| \hat \varphi^J_{Z_{t}}(u) - \fctcar{Z_{t}}(u)| 
        &\leq \1{A(c) \cap B_{h, \tau}(m)} \Big( | \hat \varphi^J_{Z_{t}}(u) -  \hat \varphi^J_{Z_{t}}(kh) |  \\
        &\qquad +  | \hat \varphi^J_{Z_{t}}(kh) -  \varphi_{Z_{t}}(kh) | + h\E[|Z_t|] \Big)
        \eqfinp
    \end{split} 
\end{equation*}
It follows from Equation~\eqref{eq:duval_5.1} that
\begin{equation*}
    \begin{split}
        \1{A(c) \cap B_{h, \tau}(m)} \sup_{|u| \leq m}| \hat \varphi^J_{Z_{t}}(u) - &\fctcar{Z_{t}}(u)|
        \leq \\
        & h(c + \E[|Z_t|]) + \1{A(c) \cap B_{h, \tau}(m)}\Big( | \hat \varphi^J_{Z_{t}}(kh) -  \varphi_{Z_{t}}(kh) | \Big) + h\E[|Z_t|]
        \eqfinp    
    \end{split}
\end{equation*}

By using the definition of  $B_{h, \tau}(m)$, we have 
\begin{align}
        \1{A(c) \cap B_{h, \tau}(m)} \sup_{|u| \leq m}| \hat \varphi^J_{Z_{t}}(u) - \fctcar{Z_{t}}(u)| 
        &\leq h(c + \E[|Z_t|]) + \tau \sqrt{\frac{\log(Jt)}{Jt}} + h\E[|Z_t|] \notag{}\\
        &\leq hc + 2h\E[|Z_t|] + \tau \sqrt{\frac{\log(Jt)}{Jt}} \notag{}\\
        &\leq hc + 2h(t\E[|X_1| + \E[|\varepsilon|]]) + \tau \sqrt{\frac{\log(Jt)}{Jt}}        
    \eqfinp    
    \label{eq:duval_5.4}
\end{align}

In particular, if $c = t$,  $h = o(\sqrt{\frac{\log(Jt)}{Jt}})$ such that $h > \frac{1}{\sqrt{Jt}}$ and $\zeta > \tau$, Equation~\eqref{eq:duval_5.4}, allows us to say that 
    \begin{equation*}
        A(c) \cap B_{h, \tau(m)} \subset \Omega_{\zeta, t}(m)
        \eqfinp
    \end{equation*}
In addition, $h > \frac{1}{\sqrt{J t}}$, \eqref{eq:duval_5.1} et \eqref{eq:duval_eq5.2}), we prove that for all 
$\eta > 0$, 
    \begin{subequations}
    \begin{align*}
        \proba(\Omega_{\zeta, t}^c(m)) 
        &\leq 
            \proba(A^c(t)) + \proba(B_{h, \tau}^c)(m) \\
        &\leq \frac{\E[X_1^2]}{J t} + \frac{\E[\noise^2]}{J t^2}  + 4 \Big\lceil \frac{m}{h} \Big\rceil (J t)^{- \frac{\tau^2}{2 t}} \\
        &\leq \frac{\E[X_1^2]}{J t} + \frac{\E[\noise^2]}{J t^2}  + 4 m (J t)^{- \frac{t - \tau^2}{2 t}}
        \eqfinp
    \end{align*}
    \end{subequations}
We obtain the result by taking $\tau^2 = t(1 + 2\eta)$.

\end{proof}

\subsection{Lemmas for multiplicative decompounding}

\begin{lemma}(Lemma 5.2. of \cite{duval2019adaptive}) \label{lm:main2}
Let $\gamma \in [0, \infty)$, $t \in [0, \infty)$ and consider
    \begin{equation*}
        M_{J, t}^{(\gamma)} = \min \Big\{ u \geq 0: |\varphi_{Z_t}(u)| = \gamma \sqrt{\log(Jt)/(Jt)} \Big\}
        \eqfinv
    \end{equation*}
    
with the convention $\inf\{\emptyset\} = \infty$.
    
Let $\zeta \in [0, \infty)$ be a positive real \st $0 < \zeta < \gamma$. Then 
    \begin{equation*}
        \1{|u| \leq M_{J,t}^\gamma \wedge m, \Omega_{\zeta, t}(m)}
        \cdot
        \Big|
        \log (\hat \varphi_{Z_t}^J(u)) - \log \varphi_{Z_t}(u)
        \Big|
        \leq
        \frac{\gamma}{\zeta} \log \Big( \frac{\gamma}{\gamma - \zeta}\Big)
        \frac{|\hat \varphi_{Z_t}^J(u) - \varphi_{Z_t}(u)|}{|\varphi_{Z_t}(u)|}
        \eqfinp
    \end{equation*}
\end{lemma}

\begin{lemma}
\label{lm:Mellin1}
Let $m > 0$ and $\zeta > 0$. We consider the event
    \begin{equation*}
        \Omega_{\zeta, \Delta}(m) = \Bigg\{ 
            \forall u \in [-m, m], |\hat{\Mcal_1[\Delta]}(u) - \Mcal_1[\Delta](u) |
            \leq
            \zeta \sqrt{\frac{\log(n \Delta)}{n \Delta}}
        \Bigg\}
        \eqfinp
    \end{equation*}
 If $\EE[\ln(X_1)^2] < \infty$, then
 \begin{equation*}
      \PP(\Omega_{\zeta, t}^c(m)) \leq \frac{\EE[\ln(X_1)^2]}{\Delta n} + 4 \frac{m}{(n \Delta)^{\eta}}
      \eqfinp
 \end{equation*}
\end{lemma}
\begin{proof}
Let $c, h, \tau \in [0, \infty)$ be some positive reals. We define the events
    \begin{align*}
        A(c) 
        &= \left\{ \bigg| \frac{1}{n} \sum_{k = 1}^n |\ln(Z_{k\Delta})| - \EE|\ln(Z_{\Delta})|  \bigg| \leq c \right\} 
        \eqfinv
        \\
        B_{h, \tau(m)} 
        &= \left\{ \forall |k| \leq \bigg\lceil \frac{m}{h} \bigg\rceil, |\hat{\Mcal_1[\Delta]}(k h) - \Mcal_1[\Delta](k h) | \leq \tau \sqrt{\frac{\log(n \Delta)}{n \Delta}}
        \right\}   \eqfinp 
     \end{align*}
    As we know that the function $u \mapsto e^{iux}$ is 1-Lipschitz, then for all $u \in \RR$ and for all $h \in [0, \infty)$ we have
    \begin{subequations}
    \begin{align*}
        \big| \hat{\Mcal_1[\Delta]}(u + h) - \hat{\Mcal_1[\Delta]}(u) \big|
        & = \Big|  \frac{1}{n} \sum_{k = 1}^{n} Z_{k\Delta}^{i(u+h)} -  \frac{1}{n} \sum_{k = 1}^{n} Z_{k\Delta}^{c - 1 + iu}  \Big| \\
        & =  \frac{1}{n} \sum_{k = 1}^{n}  \Big|  Z_{k\Delta}^{i(u+h)} -   Z_{k\Delta}^{iu}  \Big| \\
        & =  \frac{1}{n} \sum_{k = 1}^{n} \Big| e^{i(u+h)\ln(Z_{k\Delta})} -   e^{iu \ln(Z_{k\Delta})}  \Big| \\
        &\leq  \frac{h}{n}\sum_{k = 1}^n \Big| \ln(Z_{k\Delta}) \Big| \\
        &= h \Big( \frac{1}{n} \sum_{k = 1}^n \Big[ |\ln(Z_{k\Delta})| - \EE[|\ln(Z_{\Delta})|] \Big] +  \EE[|\ln(Z_{\Delta})| \Big)
        \eqfinp
    \end{align*}
    \end{subequations}
The definition of $A(c)$ ensures that for all $u \in \RR$ and $h > 0$
    \begin{equation}
        \big| \hat{\Mcal_1[\Delta]}(u + h) - \hat{\Mcal_1[\Delta]}(u) \big| \ind{A(c)}
        \leq
        h \big( c + \EE[|\ln(Z_{\Delta})| \big)
        \leq 
        h \big( c + \Delta \EE[|\ln(X_1)| \big)
        \eqfinp
    \label{eq:duval_5.1mellin}
 \end{equation}
If $\EE[\ln(X_1)^2]$ is finite, the Markov inequality leads to 
    \begin{equation}
        \PP(A(c)^c) \leq \frac{\Delta \EE[\ln(X_1)^2]}{c^2 n}
        \eqfinp
    \label{controlAc}
    \end{equation}

    Moreover we have
    \begin{align*}
        \PP(B_{h, \tau(m)}^c) 
        &=  \PP \left( \exists |k| \leq \bigg\lceil \frac{m}{h} \bigg\rceil, \Big|\hat{\Mcal_1[\Delta]}(k h) - \Mcal_1[\Delta](k h) \Big| > \tau \sqrt{\frac{\log(n \Delta)}{n \Delta}} \right) \\
        &\leq \sum_{k = - \big\lceil \frac{m}{h} \big\rceil}^{\big\lceil \frac{m}{h} \big\rceil} \PP \bigg( \Big|\hat{\Mcal_1[\Delta]}(k h) - \Mcal_1[\Delta](k h) \Big| > \tau \sqrt{\frac{\log(n \Delta)}{n \Delta}}  \bigg)  \\
        &\leq 
            \sum_{k = - \big\lceil \frac{m}{h} \big\rceil}^{\big\lceil \frac{m}{h} \big\rceil} \PP \Bigg(
           \Big|  \sum_{k = 1}^{n} Z_{k\Delta}^{i(kh)} - \EE[Z_{\Delta}^{i(kh)}]  \Big| >  \tau \sqrt{\frac{n \log(n \Delta)}{\Delta}}  \Bigg) 
           \eqfinp
    \end{align*}
    As $| Z_{\Delta}^{i(kh)}| \leq 1$ almost surely, Hoeffding's inequality ensures that 
    \begin{align*}
        \PP(B_{h, \tau(m)}^c)     
        \leq \sum_{k = - \big\lceil \frac{m}{h} \big\rceil}^{\big\lceil \frac{m}{h} \big\rceil} 2 \exp \Big( - \frac{\tau^2\log(n \Delta)}{2 t}\Big) = 4 \Big\lceil \frac{m}{h} \Big\rceil (n \Delta)^{-\tau^2 / (2 \Delta)}
        \eqfinp
    \end{align*}

Let $|u| \leq m$. There exists $k$ a positive integer such that $u \in [kh - \frac{h}{2}, kh + \frac{h}{2}]$. It follows that 
    \begin{equation*}
    \begin{split}
        \ind{A(c) \cap B_{h, \tau}(m)} |\hat{\Mcal_1[\Delta]}(u) &- \Mcal_1[\Delta](u) |
        \leq \ind{A(c) \cap B_{h, \tau}(m)} \Big( \big|\hat{\Mcal_1[\Delta]}(u) - |\hat{\Mcal_1[\Delta]}(kh) \big|  \\
        &\qquad +  \big| \hat{\Mcal_1[\Delta]}(kh) - \Mcal_1[\Delta](kh) \big| + \big| \Mcal_1[\Delta](kh) -  \Mcal_1[\Delta](u) \big| \Big)
        \eqfinp
    \end{split}
    \end{equation*}
We bound the three right terms by using respectively the Equation~\eqref{eq:duval_5.1mellin}, the definition of $B_{h, \tau}$ and the fact that $u \to e^{iux}$ is 1-Lipschitz. It follows that 
    \begin{equation*}
        \ind{A(c) \cap B_{h, \tau}(m)} |\hat{\Mcal_1[\Delta]}(u) - \Mcal_1[\Delta](u) |
        = 
        2 h \Delta \EE[|\ln(X_1)|] + hc +  \tau \sqrt{\frac{\log(n \Delta)}{n \Delta}}
        \eqfinp
    \end{equation*}

By fixing $c = \Delta$,  $h = o \Big( \sqrt{\frac{\log(n\Delta)}{n\Delta}} \Big)$ such that $h > \frac{1}{\sqrt{n\Delta}}$ and $\zeta > \tau$, it follows that
    \begin{equation*}
        A(c) \cap B_{h, \tau(m)} \subset \Omega_{\zeta, \Delta}(m)
        \eqfinp
    \end{equation*}
In addition, $h > \frac{1}{\sqrt{n \Delta}}$, \eqref{eq:duval_5.1}, we prove that for all 
$\eta > 0$, 
    \begin{align*}
        \PP(\Omega_{\zeta, t}^c(m)) 
        &= 
        \PP(A(c)^c) +  \PP(B_{h, \tau(m)}^c)  
        \leq \frac{\EE[\ln(X_1)^2]}{\Delta n} + 4 \Big\lceil \frac{m}{h} \Big\rceil (n \Delta)^{-\tau^2 / (2 \Delta)} \\
        &\leq \frac{\EE[\ln(X_1)^2]}{\Delta n} + 4 m (n \Delta)^{- \frac{t - \tau^2}{2 t}}
        \eqfinp
    \end{align*}
We obtain the result by taking $\tau^2 = t(1 + 2\eta)$.
\end{proof}

\begin{lemma}(Duval, Kappus \cite{duval2019adaptive} Lemma 5.2) \label{lm:Mellin2}
Let $\gamma \in [0, \infty)$ and consider
    \begin{equation*}
        M_{n, \Delta}^{(\gamma_\Delta)} = \min \Big\{ u \geq 0: |\Mcal_1[\Delta](u) | = \gamma c_m \sqrt{\log(Jt)/(Jt)} \Big\}
        \eqfinv
    \end{equation*}
    
with the convention $\inf\{\emptyset\} = \infty$.
    
Let $\zeta \in [0, \infty)$ be a positive real \st $0 < \zeta < \gamma$. Then 
    \begin{equation*}
        \ind{|u| \leq M_{J,t}^\gamma \wedge m, \Omega_{\zeta, t}(m)}
        \cdot
        \big|\hat{\log \Mcal_1[\Delta]}(s) - \log \Mcal_1[\Delta](s) \big|
        \leq
        \frac{\gamma}{\zeta} \log \Big( \frac{\gamma}{\gamma - \zeta}\Big)
        \frac{|\hat{\Mcal_1[\Delta]}(s) - \Mcal_1[\Delta](s)|}{|\Mcal_1[\Delta](s)|}
        \eqfinp
    \end{equation*}
\end{lemma}

\section{Acknowledgments}
This work is a part of the author's Ph.D thesis under the supervision of Marie Doumic, Marc Hoffmann and Lydia Robert. I would like to thank for their valuable remarks on this article.
The author's research is supported by a Ph.D Inria Grant. 

\bibliographystyle{alpha}
\bibliography{oneforall.bib}

\end{document}